\documentclass[reqno]{amsart}

\usepackage{amsmath,amsfonts,amssymb,amscd,verbatim,delarray,fullpage,float}

\usepackage{tikz}
\usepackage{tikz-cd}
\usetikzlibrary{matrix,arrows}
\usepackage{stmaryrd}
\usepackage{amsthm}
\usepackage{url}
\usepackage[polutonikogreek,english]{babel}

\DeclareSymbolFont{bbold}{U}{bbold}{m}{n}
\DeclareSymbolFontAlphabet{\mathbbold}{bbold}

\DeclareMathOperator{\cond}{cond}

\DeclareMathOperator{\Li}{Li}

\DeclareMathOperator{\prob}{Prob}

\DeclareMathOperator{\disc}{disc}

\DeclareMathOperator{\rad}{rad}

\DeclareMathOperator{\genus}{genus}

\DeclareMathOperator{\frob}{Frob}

\DeclareMathOperator{\GL}{GL}
\DeclareMathOperator{\SL}{SL}
\DeclareMathOperator{\PSL}{PSL}

\chardef\bslash=`\\ 

\begin{document}

\restylefloat{table}

\newtheorem{Theorem}{Theorem}[section]

\newtheorem{example}[Theorem]{Example}
\newtheorem{cor}[Theorem]{Corollary}
\newtheorem{goal}[Theorem]{Goal}

\newtheorem{Conjecture}[Theorem]{Conjecture}
\newtheorem{guess}[Theorem]{Guess}

\newtheorem{exercise}[Theorem]{Exercise}
\newtheorem{Question}[Theorem]{Question}
\newtheorem{lemma}[Theorem]{Lemma}
\newtheorem{property}[Theorem]{Property}
\newtheorem{proposition}[Theorem]{Proposition}
\newtheorem{ax}[Theorem]{Axiom}
\newtheorem{claim}[Theorem]{Claim}

\newtheorem{nTheorem}{Surjectivity Theorem}

\theoremstyle{definition}
\newtheorem{Definition}[Theorem]{Definition}
\newtheorem{problem}[Theorem]{Problem}
\newtheorem{question}[Theorem]{Question}
\newtheorem{Example}[Theorem]{Example}

\newtheorem{remark}[Theorem]{Remark}
\newtheorem{diagram}{Diagram}
\newtheorem{Remark}[Theorem]{Remark}
\newcommand{\diagref}[1]{diagram~\ref{#1}}
\newcommand{\thmref}[1]{Theorem~\ref{#1}}
\newcommand{\secref}[1]{Section~\ref{#1}}
\newcommand{\subsecref}[1]{Subsection~\ref{#1}}
\newcommand{\lemref}[1]{Lemma~\ref{#1}}
\newcommand{\corref}[1]{Corollary~\ref{#1}}
\newcommand{\exampref}[1]{Example~\ref{#1}}
\newcommand{\remarkref}[1]{Remark~\ref{#1}}
\newcommand{\corlref}[1]{Corollary~\ref{#1}}
\newcommand{\claimref}[1]{Claim~\ref{#1}}
\newcommand{\defnref}[1]{Definition~\ref{#1}}
\newcommand{\propref}[1]{Proposition~\ref{#1}}
\newcommand{\prref}[1]{Property~\ref{#1}}
\newcommand{\itemref}[1]{(\ref{#1})}
\newcommand{\ul}[1]{\underline{#1}}


\newcommand{\CE}{\mathcal{E}}
\newcommand{\CG}{\mathcal{G}}\newcommand{\CV}{\mathcal{V}}
\newcommand{\CL}{\mathcal{L}}
\newcommand{\CM}{\mathcal{M}}
\newcommand{\A}{\mathcal{A}}
\newcommand{\CO}{\mathcal{O}}
\newcommand{\B}{\mathcal{B}}
\newcommand{\CS}{\mathcal{S}}
\newcommand{\CX}{\mathcal{X}}
\newcommand{\CY}{\mathcal{Y}}
\newcommand{\CT}{\mathcal{T}}
\newcommand{\CW}{\mathcal{W}}
\newcommand{\CJ}{\mathcal{J}}

\newcommand\myeq{\mathrel{\stackrel{\makebox[0pt]{\mbox{\normalfont\tiny def}}}
{\Longleftrightarrow}}}
\newcommand{\st}{\sigma}
\renewcommand{\k}{\varkappa}
\newcommand{\Frac}{\mbox{Frac}}
\newcommand{\XC}{\mathcal{X}}
\newcommand{\wt}{\widetilde}
\newcommand{\wh}{\widehat}
\newcommand{\mk}{\medskip}
\renewcommand{\sectionmark}[1]{}
\renewcommand{\Im}{\operatorname{Im}}
\renewcommand{\Re}{\operatorname{Re}}
\newcommand{\la}{\langle}
\newcommand{\ra}{\rangle}
\newcommand{\LND}{\mbox{LND}}
\newcommand{\Pic}{\mbox{Pic}}
\newcommand{\lnd}{\mbox{lnd}}
\newcommand{\GLND}{\mbox{GLND}}\newcommand{\glnd}{\mbox{glnd}}
\newcommand{\Der}{\mbox{DER}}\newcommand{\DER}{\mbox{DER}}
\renewcommand{\th}{\theta}
\newcommand{\ve}{\varepsilon}
\newcommand{\1}{^{-1}}
\newcommand{\iy}{\infty}
\newcommand{\iintl}{\iint\limits}
\newcommand{\capl}{\operatornamewithlimits{\bigcap}\limits}
\newcommand{\cupl}{\operatornamewithlimits{\bigcup}\limits}
\newcommand{\suml}{\sum\limits}
\newcommand{\ord}{\operatorname{ord}}
\newcommand{\gal}{\operatorname{Gal}}
\newcommand{\bk}{\bigskip}
\newcommand{\fc}{\frac}
\newcommand{\g}{\gamma}
\newcommand{\be}{\beta}
\newcommand{\dl}{\delta}
\newcommand{\Dl}{\Delta}
\newcommand{\lm}{\lambda}
\newcommand{\Lm}{\Lambda}
\newcommand{\om}{\omega}
\newcommand{\ov}{\overline}
\newcommand{\vp}{\varphi}
\newcommand{\kap}{\varkappa}

\newcommand{\Vp}{\Phi}
\newcommand{\Varphi}{\Phi}
\newcommand{\BC}{\mathbb{C}}
\newcommand{\C}{\mathbb{C}}\newcommand{\BP}{\mathbb{P}}
\newcommand{\BQ}{\mathbb {Q}}
\newcommand{\BM}{\mathbb{M}}
\newcommand{\mbh}{\mathbb{H}}
\newcommand{\BR}{\mathbb{R}}\newcommand{\BN}{\mathbb{N}}
\newcommand{\BZ}{\mathbb{Z}}\newcommand{\BF}{\mathbb{F}}
\newcommand{\BA}{\mathbb {A}}
\renewcommand{\Im}{\operatorname{Im}}
\newcommand{\idd}{\operatorname{id}}
\newcommand{\ep}{\epsilon}
\newcommand{\tp}{\tilde\partial}
\newcommand{\doe}{\overset{\text{def}}{=}}
\newcommand{\supp} {\operatorname{supp}}
\newcommand{\loc} {\operatorname{loc}}
\newcommand{\de}{\partial}
\newcommand{\z}{\zeta}
\renewcommand{\a}{\alpha}
\newcommand{\G}{\Gamma}
\newcommand{\der}{\mbox{DER}}

\newcommand{\Spec}{\operatorname{Spec}}
\newcommand{\Sym}{\operatorname{Sym}}
\newcommand{\Aut}{\operatorname{Aut}}

\newcommand{\Idd}{\operatorname{Id}}

\newcommand{\tG}{\widetilde G}
\newcommand{\F}{\mathbb{F}}
\newcommand{\Q}{\mathbb{Q}}
\newcommand{\Z}{\mathbb{Z}}
\newcommand{\XG}{\text{N}_{s}(5)'}
\newcommand{\tB}{\text{B}}
\newcommand{\Gal}{\text{Gal}}
\newcommand{\cX}{\mathcal{X}}
\newcommand{\Inn}{\text{Inn}}
\newcommand{\bP}{\mathbf{P}}
\newcommand{\FX}{\mathfrac {X}}
\newcommand{\FV}{\mathfrac {V}}
\newcommand{\SX}{\mathcal {X}}
\newcommand{\SV}{\mathcal {V}}
\newcommand{\SO}{\mathcal {O}}
\newcommand{\SD}{\mathcal {D}}
\newcommand{\Sr}{\rho}
\newcommand{\SR}{\mathcal {R}}
\newcommand{\cl}{\mathcal{C}}
\newcommand{\ok}{\mathcal{O}_K}
\newcommand{\ab}{\mathcal{AB}}

\setcounter{equation}{0} \setcounter{section}{0}

\newcommand{\ds}{\displaystyle}
\newcommand{\gl}{\lambda}
\newcommand{\gL}{\Lambda}
\newcommand{\gge}{\epsilon}
\newcommand{\gG}{\Gamma}
\newcommand{\ga}{\alpha}
\newcommand{\gb}{\beta}
\newcommand{\gd}{\delta}
\newcommand{\gD}{\Delta}
\newcommand{\gs}{\sigma}
\newcommand{\mbq}{\mathbb{Q}}
\newcommand{\mbr}{\mathbb{R}}
\newcommand{\mbz}{\mathbb{Z}}
\newcommand{\mbc}{\mathbb{C}}
\newcommand{\mbn}{\mathbb{N}}
\newcommand{\mbp}{\mathbb{P}}
\newcommand{\mbf}{\mathbb{F}}
\newcommand{\mbe}{\mathbb{E}}
\newcommand{\lcm}{\text{lcm}\,}
\newcommand{\mf}[1]{\mathfrak{#1}}
\newcommand{\ol}[1]{\overline{#1}}
\newcommand{\mc}[1]{\mathcal{#1}}
\newcommand{\mb}[1]{\mathbb{#1}}
\newcommand{\nequiv}{\equiv\hspace{-.07in}/\;}
\newcommand{\bnequiv}{\equiv\hspace{-.13in}/\;}

\title{On the acyclicity of reductions of elliptic curves modulo primes in arithmetic progressions}
\author{Nathan Jones and Sung Min Lee}

\renewcommand{\thefootnote}{\fnsymbol{footnote}} 
\renewcommand{\thefootnote}{\arabic{footnote}}

\renewcommand{\thefootnote}{\fnsymbol{footnote}} 
\renewcommand{\thefootnote}{\arabic{footnote}} 

\date{}

\begin{abstract}
Let $E$ be an elliptic curve defined over $\mbq$ and, for a prime $p$ of good reduction for $E$ let $\tilde{E}_p$ denote the reduction of $E$ modulo $p$.  Inspired by an elliptic curve analogue of Artin's primitive root conjecture posed by S. Lang and H. Trotter in 1977, J-P. Serre adapted methods of C. Hooley to prove a GRH-conditional asymptotic formula for the number of primes $p \leq x$ for which the group $\tilde{E}_p(\mbf_p)$ is cyclic.  More recently, Akbal and G\"{u}lo$\breve{\text{g}}$lu considered the question of cyclicity of $\tilde{E}_p(\mbf_p)$ under the additional restriction that $p$ lie in an arithmetic progression.  In this note, we study the issue of which arithmetic progressions $a \bmod n$ have the property that, for all but finitely many primes $p \equiv a \bmod n$, the group $\tilde{E}_p(\mbf_p)$ is \emph{not} cyclic, answering a question of Akbal and G\"{u}lo$\breve{\text{g}}$lu on this issue.  
\end{abstract}

\maketitle

\section{Introduction and statement of results} \label{introduction}
Let $E$ be an elliptic curve defined over $\mbq$ and let $p$ be a rational prime of good reduction for $E$, i.e. assume that the reduction $\tilde{E}_p$ of $E$ modulo $p$, obtained by reducing a minimal Weierstrass model of $E$ modulo $p$, is non-singular. Then $\tilde{E}_p$ is itself an elliptic curve over the finite field $\mbf_p$, and we may consider the group $\tilde{E}_p(\mbf_p)$ of $\mbf_p$-rational points of $\tilde{E}_p$.

There are various open questions surrounding the nature of the abelian groups $\tilde{E}_p(\mbf_p)$, as $p$ varies. One such example is the following conjecture about the cyclicity of $\tilde{E}_p(\mbf_p)$, which underlies an elliptic curve analogue of Artin's primitive root conjecture that was proposed by S. Lang and H. Trotter in 1977 \cite{langtrotter}.
\begin{Conjecture} \label{cyclicityconjecture}
Let $E$ be an elliptic curve over $\mbq$ of conductor $N_E$.  There is a constant $C_E \geq 0$ such that
\begin{equation} \label{cyclicityasymptotic}
\lim_{x \to \infty} \frac{ \left| \{ p \leq x : p \nmid N_E, \; \tilde{E}_p(\mbf_p) \text{ is cyclic} \} \right| }{ \left| \{ p \leq x \} \right| } = C_E.
\end{equation}
\end{Conjecture}
In fact, the constant $C_E$ is given explicitly by
\[
C_E = \sum_{n\geq 1} \frac{\mu(n)}{[\Q(E[n]):\Q]},
\]
where $\mu(\cdot)$ denotes the M\"{o}bius function and $\mbq(E[n])$ the $n$-division field of $E$.

Regarding the positivity of $C_E$, J-P. Serre \cite[pp. 465--466]{serreo} observed that
\begin{equation} \label{conditionforCEequals0}
C_E = 0 \; \Longleftrightarrow \; E[2] \subseteq E(\mbq)
\end{equation}
(a proof of this may be found in \cite[p. 619]{cojocarumurty}).
When \eqref{conditionforCEequals0} holds, the embedding $\mbz/2\mbz \times \mbz/2\mbz \simeq E[2] = E(\mbq)[2] \hookrightarrow \tilde{E}_p(\mbf_p)$ (see for instance \cite[VII, Proposition 3.1]{silverman}), valid for every odd prime $p \nmid N_E$, shows that
\[
C_E = 0 \; \Longrightarrow \; \left| \{ p \text{ prime} \, : p \nmid N_E, \; \tilde{E}_p(\mbf_p) \text{ is cyclic} \} \right| \leq 1. 
\]
It is well-known that $\tilde{E}_p(\mbf_p) \simeq \mbz/d_p(E)\mbz \times \mbz/e_p(E)\mbz$ for certain $d_p(E), e_p(E) \in \mbn$ with $d_p(E)$ dividing $e_p(E)$; obviously this group is cyclic if and only if its size $| \tilde{E}_p(\mbf_p)|$ is equal to its exponent $e_p(E)$.  The quantities $d_p(E)$ and $e_p(E)$ have been extensively studied (see \cite{banksetal}, \cite{davidetal}, \cite{duke}, \cite{freibergkurlberg} and the references therein).
The question of cyclicity of $\tilde{E}_p(\mbf_p)$ seems to have first appeared in a paper by I. Borosh, C. J. Moreno, and H. Porta in 1975 \cite{bmp}, which calculates the structures of $\tilde{E}_p(\mbf_p)$ for various elliptic curves $E$ and many primes $p$, and expresses a version of Conjecture \ref{cyclicityconjecture} without completely determining the constant $C_E$.

In 1979, J-P. Serre \cite[pp. 465--468]{serreo} proved Conjecture \ref{cyclicityconjecture}, assuming the Generalized Riemann Hypothesis (denoted GRH) by adapting techniques from C. Hooley's GRH-conditional proof of Artin's primitive root conjecture \cite{hooley} (a detailed proof can be found in \cite[pp. 160--161]{murty1}).  In 1983, M. Ram Murty \cite{murty1} gave an unconditional proof for elliptic curves with complex multiplication (denoted CM).  In 2002, A.C. Cojocaru proved a non-CM version of Conjecture \ref{cyclicityconjecture} under a weaker hypothesis than GRH \cite{cojocaru}. Specifically, Cojocaru proved the bound $\left| \{ p \leq x : p \nmid N_E, \; \tilde{E}_p(\mbf_p) \text{ is cyclic} \} \right| - C_E | \{ p \leq x \} | = O(x \log \log x / \log^2 x)$ under a quasi-GRH, i.e. assuming that the Dedekind zeta functions of the division fields of $E$ do not vanish for $\Re(s) > 3/4$ (further improvements on this remainder term may be found in \cite[Theorem 45]{cojocarusynopsis}).
In addition to the above-mentioned results, Conjecture \ref{cyclicityconjecture} has also been proved ``on average'' over elliptic curves $E$ of bounded height \cite{banksshparlinski}.  Finally, Conjecture \ref{cyclicityconjecture} has recently been considered in the more general context of elliptic curves over arbitrary number fields \cite{campagnastevenhagen}, in which case the question of vanishing of the conjectural density becomes more delicate.  

Inspired by Conjecture \ref{cyclicityconjecture}, Y. Akbal and A. M. G\"{u}lo$\breve{\text{g}}$lu \cite{akbalguloglu} considered the question of cyclicity of $\tilde{E}_p(\mbf_p)$ for the subset of those primes $p$ which lie in a fixed arithmetic progression (this question was also considered in the Ph.D. dissertation of J. Brau \cite{brau}).  Specifically, for any fixed $a, n \in \mbn$ with $\gcd(a,n) = 1$, let us define the counting function $\pi_{E,a,n}(x)$ by
\begin{equation} \label{defofpisubEanofx}
\pi_{E,a,n}(x) := \left| \left\{ p \leq x : p \nmid N_E, \, p \equiv a \bmod n \; \text{ and } \; \tilde{E}_p(\mbf_p) \; \text{ is cyclic} \right\} \right|.
\end{equation}
Akbal and G\"{u}lo$\breve{\text{g}}$lu proved the following theorem, wherein the constant $C_{E,a,n} \geq 0$ is given by
\begin{equation} \label{CsubEan}
C_{E,a,n} := \sum_{d = 1}^\infty \frac{\mu(d) \gamma_{a,n}(\mbq(E[d]))}{[ \mbq(E[d]) \mbq(\zeta_n) : \mbq]}.
\end{equation}
In this definition, $\zeta_n$ denotes a primitive $n$-th root of unity, and
\begin{equation} \label{defofgammasuban}
\gamma_{a,n}\left( \mbq(E[d]) \right) := 
\begin{cases}
1 & \text{ if } \gs_a \text{ fixes } \mbq(E[d]) \cap \mbq(\zeta_n) \text{ pointwise} \\
0 & \text{ otherwise,}
\end{cases}
\end{equation}
where $\gs_a \in \gal(\mbq(\zeta_n)/\mbq)$ refers to the unique automorphism satisfying $\gs_a(\zeta_n) = \zeta_n^a$. 
\begin{Theorem} \label{akbalgulogluthm} (\cite[Theorems 3 and 6]{akbalguloglu})  Let $E$ be an elliptic curve defined over $\mbq$ and assume that, if $E$ has CM, then it has CM by the full ring of integers of an imaginary quadratic field. Furthermore, fix $n \in \mbn$ and, for each square-free $d \geq 1$, assume the Generalized Riemann Hypothesis for the Dedekind zeta function of the field $\mbq(E[d])\mbq(\zeta_n)$.  For any $a \in \mbz$ with $\gcd(a,n) = 1$, we then have
\[
\pi_{E,a,n}(x) = C_{E,a,n} \Li(x) + O_{E,n}\left( x^{5/6} ( \log x )^{2/3} \right),
\]
where $\Li(x) := \int_2^x \frac{1}{\log t} dt$ denotes the logarithmic integral.
\end{Theorem}
In particular, under GRH, Akbal and G\"{u}lo$\breve{\text{g}}$lu proved that $\pi_{E,a,n}(x) \sim C_{E,a,n} \Li(x)$ as $x \to \infty$, provided the constant $C_{E,a,n}$ is positive.  Furthermore, they noted \cite[p. 3]{akbalguloglu} that
\begin{equation} \label{sufficientforfinitelymanyprimes}
\exists \text{ a prime } \ell \text{ such that } \mbq(E[\ell]) \subseteq \mbq(\zeta_n) \text{ and } \gs_a |_{\mbq(E[\ell])} \equiv 1 \; \Longrightarrow \; \lim_{x \to \infty} \pi_{E,a,n}(x) < \infty.
\end{equation}
(The left-hand condition above implies that $C_{E,a,n} = 0$, so the implication \eqref{sufficientforfinitelymanyprimes} is consistent with Theorem \ref{akbalgulogluthm}.)  They then posed the following question (see \cite[Question 1]{akbalguloglu}).
\begin{question} \label{akbalgulogluqn}
Let $E$ be an elliptic curve over $\mbq$ and assume the notation as above.  Is the converse of \eqref{sufficientforfinitelymanyprimes} true?
\end{question}
The purpose of this note is to answer Question \ref{akbalgulogluqn} in the negative via a concrete example (see Example \ref{mainthmspecial} and Theorem \ref{mainthm} below), and to propose the following biconditional analogue of \eqref{sufficientforfinitelymanyprimes}, giving a precise (conjectural) interpretation of when $\ds \lim_{x \to \infty}\pi_{E,a,n}(x) < \infty$ in terms of division fields of $E$:
\begin{equation} \label{refinedsufficientforfinitelymanyprimes}
\begin{pmatrix} \exists d \in \mbn_{\geq 2} \text{ such that } \forall \gs \in \gal\left(\mbq(E[d])\mbq(\zeta_n)/\mbq\right) \text{ with} \\ \gs |_{\mbq(\zeta_n)} = \gs_a, \, \exists \text{ a prime } \ell \mid d \text{ for which }  \gs |_{\mbq(E[\ell])} = 1 \end{pmatrix} \; \Longleftrightarrow \; \lim_{x \to \infty} \pi_{E,a,n}(x) < \infty.
\end{equation}
Note that the condition on the left-hand side of \eqref{sufficientforfinitelymanyprimes} implies the condition on the left-hand side of \eqref{refinedsufficientforfinitelymanyprimes}.  
In the following example, the former is false while the latter is true (with $d = 6$).
\begin{Example} \label{mainthmspecial}
Let $E$ be the (non-CM) elliptic curve over $\mbq$ defined by the Weierstrass equation
\[
E : \; y^2 + xy + y = x^3 + 32271697x - 1200056843302.
\]
The conductor of $E$ is $2 \cdot 3 \cdot 5 \cdot 7 \cdot 11 \cdot 31$; its LMFDB label is $71610.s6$ and its Cremona label is $71610s4$.  Furthermore, we have:
\begin{enumerate}
\item the counting function $\pi_{E,3,8}(x)$ defined by \eqref{defofpisubEanofx} satisfies $\pi_{E,3,8}(x) = 0$ for all $x \geq 0$;
\item for each prime $\ell$, $\mbq(E[\ell]) \nsubseteq \mbq(\zeta_8)$.
\end{enumerate}
\end{Example}
We remark that, for any elliptic curve $E$ over $\mbq$ that satisfies the left-hand condition in \eqref{sufficientforfinitelymanyprimes}, $\tilde{E}_p(\mbf_p)$ fails to be cyclic for good primes $p \equiv a \bmod{n}$ because there exists a prime $\ell$ such that $\mbz/\ell\mbz \times \mbz/\ell\mbz \subseteq \tilde{E}_p(\mbf_p)$ for every such $p$. In contrast, for the elliptic curve $E$ of Example \ref{mainthmspecial}, the reason that $\tilde{E}_p(\mbf_p)$ is not cyclic for any good prime $p \equiv 3 \bmod{8}$ is that, for any such prime, either $\mbz/2\mbz \times \mbz/2\mbz \subseteq \tilde{E}_p(\mbf_p)$ or $\mbz/3\mbz \times \mbz/3\mbz \subseteq \tilde{E}(\mbf_p)$.  For more details, see Remark \ref{exampleonepointfourremark} below.

We propose the following refinement of \cite[Conjecture 3]{akbalguloglu}.
\begin{Conjecture}
Let $E$ be an elliptic curve over $\mbq$ of conductor $N_E$ and let $a, n \in \mbn$ with $\gcd(a,n) = 1$.  Suppose that the condition on the left-hand side of \eqref{refinedsufficientforfinitelymanyprimes} does {\textbf{not}} hold. Then there are infinitely many primes $p \equiv a \bmod{n}$ for which $\tilde{E}_p(\mbf_p)$ is cyclic.  Furthermore, the constant $C_{E,a,n}$, defined by \eqref{CsubEan}, is positive, and we have
\begin{equation} \label{pisubEanasymptoticinconj}
\pi_{E,a,n}(x) \sim C_{E,a,n} \Li(x)
\end{equation}
as $x \to \infty$, where $\pi_{E,a,n}(x)$ is as in \eqref{defofpisubEanofx}.   
\end{Conjecture}
\begin{remark}
To clarify what is known and what is conjectured: given Theorem \ref{akbalgulogluthm}, the asymptotic formula \eqref{pisubEanasymptoticinconj} is already known to be a consequence of the GRH (assuming in the CM case that $E$ has CM by the maximal order $\mc{O}_K$).  We will show in Section \ref{Csubansection} that
\begin{equation*} 
\begin{pmatrix} \exists d \in \mbn_{\geq 2} \text{ such that } \forall \gs \in \gal\left(\mbq(E[d])\mbq(\zeta_n)/\mbq\right) \text{ with} \\ \gs |_{\mbq(\zeta_n)} = \gs_a, \, \exists \text{ a prime } \ell \mid d \text{ for which }  \gs |_{\mbq(E[\ell])} = 1 \end{pmatrix} \; \Longleftrightarrow \; C_{E,a,n} = 0.
\end{equation*}
The implication $C_{E,a,n} = 0 \; \Rightarrow \; \lim_{x \to \infty} \pi_{E,a,n}(x) < \infty$ is a straightforward application of the Chebotarev density theorem, and the reverse implication $\lim_{x \to \infty} \pi_{E,a,n}(x) < \infty \; \Rightarrow \; C_{E,a,n} = 0$ is conjectural, but is known conditionally on GRH via Theorem \ref{akbalgulogluthm} (with the above-mentioned proviso in the CM case).
\end{remark}

The elliptic curve $E$ in Example \ref{mainthmspecial} belongs to an infinite family associated to a genus zero modular curve.  Viewed another way, it is a specialization of an elliptic curve $\mathbb{E}$ defined over $\mbq(t,d)$, where $t$ and $d$ are variables, infinitely many of whose specializations satisfy the conditions (1) and (2) of Example \ref{mainthmspecial}, with an appropriately chosen $a \in \mbz$ and $n \in \mbn$ replacing $3$ and $8$, respectively.  Let the polynomials $f(t), g(t) \in \mbq[t]$ be defined by
\begin{equation} \label{defoffoft}
\begin{split}
f(t) &:= 16t^6 - 24t^4 - 8t^3 + 36t^2 + 6t + 1, \\
g(t) &:= 64t^8 + 64t^7 + 64t^6 - 128t^5 - 56t^4 + 16t^3 + 64t^2 - 8t + 1.
\end{split}
\end{equation}
Next, define the Weierstrass coefficients $a_4(t), a_6(t) \in \mbq(t)$ by 
\begin{equation} \label{defofa4anda6}
\begin{split}
a_4(t) &:= \frac{-108\left( 4t^3 - 1 \right)^3 \left( 4t^3 + 6t - 1 \right)^3 f(t)^3}{\left( 2t^2 + 2t - 1 \right)^2 \left( 4t^4 - 4t^3 + 6t^2 + 2t + 1 \right)^2 \left( 8t^4 - 8t^3 - 8t - 1 \right)^2 g(t)^2}, \\
a_6(t) &:= \frac{-432\left( 4t^3 - 1 \right)^3 \left( 4t^3 + 6t - 1 \right)^3 f(t)^3}{\left( 2t^2 + 2t - 1 \right)^2 \left( 4t^4 - 4t^3 + 6t^2 + 2t + 1 \right)^2 \left( 8t^4 - 8t^3 - 8t - 1 \right)^2 g(t)^2},
\end{split}
\end{equation}
where $f(t)$ and $g(t)$ are as in \eqref{defoffoft}. 
Furthermore, define the elliptic curve $\mb{E}$ over $\mbq(t,d)$ by the Weierstrass equation
\begin{equation} \label{defofEoft}
\mb{E} : \; y^2 = x^3 + d^2a_4(t) x + d^3a_6(t).
\end{equation}
The $j$-invariant $j_{\mb{E}}(t)$ and discriminant $\gD_{\mb{E}}(t)$ of $\mb{E}$ are given by
\begin{equation} \label{defofjofy}
\begin{split}
j_{\mb{E}}(t) &:= \frac{\left(4t^3 - 1\right)^3 \left(4t^3 + 6t - 1 \right)^3 f(t)^3}{t^3\left( t - 1 \right)^3 \left( 2t + 1 \right)^6 \left( t^2 + t + 1 \right)^3 \left( 4t^2 - 2t + 1 \right)^6}, \\
\gD_{\mb{E}}(t) &:= \frac{2^{18}3^{12}d^6t^3(t-1)^3(2t+1)^6(t^2+t+1)^3(4t^2-2t+1)^6(4t^3-1)^{6}(4t^3+6t-1)^{6}f(t)^{6}}{\left( 2t^2 + 2t - 1 \right)^{6} \left( 4t^4 - 4t^3 + 6t^2 + 2t + 1 \right)^{6} \left( 8t^4 - 8t^3 - 8t - 1 \right)^{6} g(t)^{6}}.
\end{split}
\end{equation}
Let $h_2(t), h_3(t) \in \mbq(t)$ be defined by
\begin{equation} \label{defofh2andh3}
    \begin{split}
        h_2(t) &:= t(t-1)(t^2+t+1), \\
        h_3(t) &:= \frac{6(4t^3-1)(4t^3+6t-1)(2t^2+2t-1)f(t)g(t)}{(4t^4-4t^3+6t^2+2t+1)(8t^4-8t^3-8t-1)},
    \end{split}
\end{equation}
where $f(t)$ and $g(t)$ are as in \eqref{defoffoft}.
A computation involving the appropriate division polynomials associated to $\mbe$ demonstrates that
\begin{equation} \label{divisionfieldsat2and3general}
    \mbq(t,d)\left( \mbe[2] \right) = \mbq(t,d)\left(\sqrt{h_2(t)}\right), \qquad
    \mbq(t,d)\left( \mbe[3] \right) = \mbq(t,d)\left(\sqrt{dh_3(t)}, \sqrt{-3} \right).
\end{equation}

For any $t_0 \in \mbq - \{ 0, 1, -1/2 \}$ and $d_0 \in \mbq - \{ 0 \}$, we may consider the specialization $\mbe_{t_0,d_0}$ of $\mbe$ at $(t_0,d_0)$, which is an elliptic curve over $\mbq$.  Furthermore, \eqref{divisionfieldsat2and3general} specializes to
\begin{equation} \label{divisionfieldsat2and3special}
\mbq(\mbe_{t_0,d_0}[2]) = \mbq\left(\sqrt{h_2(t_0)}\right), \qquad
    \mbq(\mbe_{t_0,d_0}[3]) = \mbq\left(\sqrt{d_0h_3(t_0)}, \sqrt{-3} \right).
\end{equation}
\sloppy In particular, we see that $\mbq(\mbe_{t_0,d_0}[3])$ is either equal to $\mbq(\sqrt{-3})$ or is a biquadratic extension of $\mbq$ which contains $\mbq(\sqrt{-3})$ as a subfield.  In the latter case, $\mbq(\mbe_{t_0,d_0}[3])$ contains three quadratic subfields, exactly two of which are ramified at the prime $3$; in either case let us denote by
\[
n_3(t_0,d_0) := 
\begin{cases}
d_0h_3(t_0) & \text{ if $3$ is unramified in } \mbq\left( \sqrt{d_0 h_3(t_0)} \right), \\
-3d_0h_3(t_0) & \text{ otherwise.}
\end{cases}
\]
Thus, $\{ \mbq(\sqrt{n_3(t_0,d_0)} ), \mbq(\sqrt{-3}), \mbq( \sqrt{-3n_3(t_0,d_0)} ) \}$ is the set of all quadratic subfields of $\mbq(\mbe_{t_0,d_0}[3])$ and 
\[
\left\{ \mbq\left( \sqrt{-3} \right), \mbq\left( \sqrt{-3n_3(t_0,d_0)} \right) \right\}
\]
is the subset of those quadratic subfields that are ramified at $3$.  Finally, we define $n_0 \in \mbn$ by
\begin{equation} \label{defofnsub0}
    n_0 := 
        \lcm \left( \left| \disc \mbq\left( \sqrt{-3h_2(t_0)} \right) \right|, \left| \disc \mbq\left( \sqrt{-3n_3(t_0,d_0)h_2(t_0)} \right) \right| \right).
\end{equation}
Note that, in case $\mbq\left( \mbe_{t_0,d_0}[3] \right) = \mbq(\sqrt{-3})$, we simply have $\mbq\left( \sqrt{-3n_3(t_0,d_0)} \right) = \mbq(\sqrt{-3})$ and $n_0 = \left| \disc \mbq\left( \sqrt{-3h_2(t_0)} \right) \right|$.  In either case, $n_0$ is the smallest positive integer for which the containment
\begin{equation} \label{twoquadraticfieldcontainments}
\mbq\left( \sqrt{-3h_2(t_0)}, \sqrt{-3n_3(t_0,d_0)h_2(t_0)} \right) \subseteq \mbq\left( \zeta_{n_0} \right)
\end{equation}
holds.  We will prove the following theorem.
\begin{Theorem} \label{mainthm}
Let the elliptic curve $\mb{E}$ over $\mbq(t,d)$ be given by \eqref{defofEoft} and let $h_2(t), h_3(t) \in \mbq(t)$ be as in \eqref{defofh2andh3}, so that \eqref{divisionfieldsat2and3general} holds. \\

\noindent (a) For any $t_0 \in \mbq - \{ 0, 1, -1/2 \}$ and $d_0 \in \mbq - \{ 0 \}$, the specialization $\mbe_{t_0,d_0}$ of $\mbe$ at $(t_0,d_0)$ is an elliptic curve over $\mbq$.  For any elliptic curve $E$ defined over $\mbq$ satisfying $j_E \notin \{ 0, 1728 \}$, $E$ satisfies
\[
[\mbq(E[2]) : \mbq] \leq 2 \quad \text{ and } \quad [\mbq(E[3]) : \mbq] \leq 4
\]
if and only if $E$ is isomorphic over $\mbq$ to a specialization $\mbe_{t_0,d_0}$ for some $t_0 \in \mbq - \{ 0, 1, -1/2 \}$ and $d_0 \in \mbq - \{ 0 \}$. \\

\noindent (b) Suppose $t_0 \in \mbq - \{ 0, 1, -1/2 \}$ and $d_0 \in \mbq - \{ 0 \}$ are chosen so that
\begin{equation} \label{conditiononQE2}
\mbq \left( \mbe_{t_0,d_0}[2] \right) \not\subseteq \mbq\left( \mbe_{t_0,d_0}[3] \right).
\end{equation}
Let $\left\{ \mbq\left( \sqrt{-3} \right), \mbq\left( \sqrt{-3n_3(t_0,d_0)} \right) \right\}$ denote the set consisting of every quadratic subfield of $\mbq(\mbe_{t_0,d_0}[3])$ in which the prime $3$ ramifies.  Define $n_0 \in \mbn$ by \eqref{defofnsub0} (note that \eqref{twoquadraticfieldcontainments} then holds) and let $a_0 \in \mbz$ be any integer coprime to $n_0$ such that the automorphism $\gs_{a_0} \in \gal\left( \mbq(\zeta_{n_0})/\mbq\right)$ satisfies
\begin{equation} \label{defofasub0}
\gs_{a_0}\left( \sqrt{-3h_2(t_0)} \right) = - \sqrt{-3h_2(t_0)} \; \text{ and } \; \gs_{a_0}\left( \sqrt{-3n_3(t_0,d_0)h_2(t_0)} \right) = - \sqrt{-3n_3(t_0,d_0)h_2(t_0)}.
\end{equation}
(Note that, by \eqref{conditiononQE2}, neither $\sqrt{-3h_2(t_0)}$ nor $\sqrt{-3n_3(t_0,d_0)h_2(t_0)}$ are in $\mbq$.) We then have that, for each prime $p \geq 5$ of good reduction for $\mbe_{t_0,d_0}$ with $p \equiv a_0 \bmod n_0$, the group $(\widetilde{\mbe_{t_0,d_0}})_p(\mbf_p)$ is not cyclic. In particular, for each $x \geq 0$, we have $\pi_{\mbe_{t_0,d_0},a_0,n_0}(x) \leq 2$. \\

\noindent (c) Suppose further that $t_0 \in \mbq - \{ 0, 1, -1/2 \}$ and $d_0 \in \mbq - \{ 0 \}$ are chosen so that $3$ ramifies in $\mbq\left( \sqrt{h_2(t_0)} \right)$ and $5$ does not ramify in $\mbq \left( \sqrt{-3h_2(t_0)}, \sqrt{-3n_3(t_0,d_0)h_2(t_0)} \right)$.  Then, for each prime $\ell$, we have $\mbq(\mb{E}_{t_0,d_0}[\ell]) \nsubseteq \mbq(\zeta_{n_0})$.
\end{Theorem}
Before proving Theorem \ref{mainthm}, we first observe that there exist infinitely many values $t_0 \in \mbq - \{ 0, 1, -1/2 \}$ for which there exists $d_0 \in \mbq - \{ 0 \}$ such that the specialized curve $\mbe_{t_0,d_0}$ satisfies the conditions of parts (b) and (c), and thus infinitely many $j$-invariants corresponding to elliptic curves $E$ over $\mbq$ that answer Question \ref{akbalgulogluqn} in the negative.
\begin{cor}
There are infinitely many $j_E \in \mbq$, where each $j_E$ is the $j$-invariant of an elliptic curve $E$ over $\mbq$ for which there exist $a \in \mbz$ and $n \in \mbn$ such that, for each $x \geq 0$, $\pi_{E,a,n}(x) \leq 2$ in spite of the fact that, for each prime $\ell$, $\mbq(E[\ell]) \not\subseteq \mbq(\zeta_n)$. 
\end{cor}
\begin{proof}
A straightforward calculation shows that, if $u_0 \in \mbq \cap \mbz_3 \cap \mbz_5 \cap \mbz_7$ and $t_0 := 21(15u_0 - 1)$, then the $3$-adic, $5$-adic and $7$-adic valuations of $h_2(t_0) = t_0 ( t_0-1 ) (t_0^2 + t_0 + 1)$ are given by
\[
v_3\left( h_2(t_0) \right) = 1, \qquad
v_5\left( h_2(t_0) \right) = 0, \qquad
v_7\left( h_2(t_0) \right) = 1.
\]
\sloppy Thus, $3$ and $7$ are ramified and $5$ is unramified in $\mbq\left( \sqrt{h_2(t_0)} \right)$, and in particular this implies that
\begin{equation} \label{5isunramified}
    \text{$5$ is unramified in the field $\mbq\left( \sqrt{-3h_2(t_0)} \right)$.}
\end{equation}  

Next, choosing $d_0 \in \mbq$ so that
\[
v_3\left( d_0 h_3(t_0) \right) = 0, \quad
v_5\left( d_0 h_3(t_0) \right) = 0, \quad
v_7\left( d_0 h_3(t_0) \right) = 0,
\]
the primes $3$, $5$ and $7$ are unramified in $\mbq\left( \sqrt{d_0h_3(t_0)}\right)$.  It follows from this and \eqref{divisionfieldsat2and3special} that $7$ is unramified in $\mbq\left( \mbe_{t_0,d_0}[3] \right)$ and ramified in $\mbq\left( \mbe_{t_0,d_0}[2]\right)$, and so \eqref{conditiononQE2} holds.  Furthermore, we have that $\mbq\left(\sqrt{-3n_3(t_0,d_0)}\right) = \mbq\left( \sqrt{-3d_0h_3(t_0)} \right)$, and this field is evidently unramified at $5$. It follows from this and \eqref{5isunramified} that $5$ is unramified in the field $\mbq\left( \sqrt{-3h_2(t_0)}, \sqrt{-3n_3(t_0,d_0)h_2(t_0)} \right)$.  
Thus, the conditions of parts (b) and (c) of Theorem \ref{mainthm} are satisfied, and so the specialized curve $E := \mbe_{t_0,d_0}$, together with the numbers $a_0 \in \mbz$ and $n_0 \in \mbn$ as described in part (b), furnish an example with $\pi_{E,a_0,n_0}(x) \leq 2$ for every $x \geq 0$ even though $\mbq(E[\ell]) \not\subseteq \mbq(\zeta_{n_0})$ for every prime $\ell$.
\end{proof}
\begin{remark} \label{exampleonepointfourremark}
Taking $t_0 = 3/5$ and $d_0 = -28910265879522405941333082$, we see that the specialization $\mbe_{t_0,d_0}$ is isomorphic over $\mbq$ to the elliptic curve $E$ of Example \ref{mainthmspecial}.  For this specialization, we have
\[
\mbq\left(E[2]\right) = \mbq\left( \sqrt{h(t_0)} \right) = \mbq(\sqrt{-6}), \quad \mbq\left(E[3]\right) = \mbq\left( \sqrt{d_0h_3(t_0)},\sqrt{-3} \right) = \mbq(\sqrt{-3}).
\]
In particular, $\mbq(E[2]) \not\subseteq \mbq(E[3])$.  Furthermore, taking $n = 8$ and $a = 3$, we see that $\mbq(E[6]) \cap \mbq(\zeta_8) = \mbq(\sqrt{2})$, and  \eqref{defofasub0} simply becomes
\[
\gs_3( \sqrt{2} ) = -\sqrt{2}.
\]
Thus, any $\gs \in \gal\left( \mbq(E[6])/\mbq \right)$ whose restriction to $\mbq(E[6]) \cap \mbq(\zeta_8)$ agrees with that of $\gs_3$ must act trivially either on $\mbq(E[2])$ or on $\mbq(E[3])$, so the left-hand condition in \eqref{refinedsufficientforfinitelymanyprimes} holds with $d=6$, $n = 8$ and $a=3$.  Since $2$ and $3$ divide $N_E$, we have that $\pi_{E,3,8}(x) = 0$ for all $x \geq 0$ (similar reasoning shows that $\pi_{E,5,8}(x) = 0$).  Finally, $3$ ramifies in $\mbq(E[2])$ and $5$ does not ramify in $\mbq(E[6])$, so by part (c) of Theorem \ref{mainthm}, we have $\mbq(E[d]) \not\subseteq \mbq(\zeta_8)$ for each $d \geq 2$.
\end{remark}
\begin{remark}
For any elliptic curve $E$ over $\mbq$ and $a, n \in \mbn$ with $\gcd(a,n) = 1$, Theorems 1 and 2 of \cite{akbalguloglu} give, for any $A \geq 2$, an unconditional lower bound
\begin{equation} \label{lowerbound}
\frac{x}{(\log x)^A} \ll \pi_{E,a,n}(x),
\end{equation}
provided $\gcd(a-1,n)$ is a power of two and none of the following three properties is fulfilled:
\begin{enumerate}
    \item The containment $\mbq(E[2]) \subseteq \mbq(\zeta_n)$ holds and $\gs_a$ fixes $\mbq(E[2])$ pointwise.
    \item The containments $\mbq \subsetneq \mbq(\sqrt{\gD_E}) \subseteq \mbq(\zeta_n)$ hold and $\gs_a$ fixes $\mbq(\sqrt{\gD_E})$ pointwise.
    \item The conductor $n_2$ of the field $\mbq(\sqrt{\gD_E})$ satisfies $n_2 = 3 \cdot \gcd(n_2,n)$ and $\chi_{-\gd_2/3}(a) = -1$, where $\gd_2$ denotes the discriminant of $\mbq(\sqrt{\gD_E})$.
\end{enumerate}
These three properties, which represent cases which are excluded from the proof of \eqref{lowerbound}, are not mutually exclusive.  As we observed earlier, if $E$ satisfies property (1) then $C_{E,a,n} = 0$, and so \eqref{lowerbound} cannot possibly hold in this case (this was also observed in \cite[Proposition 1]{akbalguloglu}).  What about elliptic curves satisfying properties (2) or (3) but not (1)?
Every specialization $E = \mbe_{t_0,d_0}$ satisfying the conditions in part (b) of Theorem \ref{mainthm} fulfills at least one of the three properties; those specialization which additionally satisfy the conditions of part (c) fail to fulfill properties (1) and (2), and (as can be checked independently) do indeed fulfill property (3), this being the novelty of the family $\mbe$.  Any elliptic curve $E$ which fulfills property (2) but not property (1) satisfies $[\mbq(E[2]) : \mbq] = 6$.  In this case, whenever $C_{E,a,n} = 0$, the integer $d$ on the left-hand side of \eqref{refinedsufficientforfinitelymanyprimes} may be taken to be odd, so that the level $2$ has nothing to do with the vanishing of $C_{E,a,n}$.
\end{remark}
 
It is natural to wonder whether or not there are any other examples of elliptic curves $E$ over $\mbq$ for which $C_{E,a,n} = 0$ for some relatively prime pair $a, n \in \mbn$ in spite of the left-hand condition in \eqref{sufficientforfinitelymanyprimes} failing to hold, besides specializations of $\mbe$.  By couching this problem in terms of modular curves, translating the condition that $C_{E,a,n} = 0$ into group-theoretical information about $\rho_E(G_\mbq)$ and performing a computer calculation, we are also able to prove the following theorem, which partially addresses the question.  
\begin{Definition}
If $E$ is an elliptic curve over $\mbq$ satisfying a given property $P$, we say that \textbf{\emph{$E$ belongs to an infinite modular curve family with property $P$}} if $j_E \in j_{\tilde{G}}(X_{\tilde{G}}(\mbq))$ for some open subgroup $\tilde{G} \subseteq \GL_2(\hat{\mbz})$ for which the modular curve $X_{\tilde{G}}$ has infinitely many rational points, and for all but finitely many $j$-invariants $j' \in j_{\tilde{G}}(X_{\tilde{G}}(\mbq))$, there exists an elliptic curve $E'$ over $\mbq$ satisfying property $P$ with $j_{E'} = j'$ (for more details about modular curves, see Section \ref{preliminariesonmodularcurvessection}).
\end{Definition}
\begin{Theorem} \label{secondthm}
Let $E$ be an elliptic curve over $\mbq$ that belongs to an infinite modular curve family with the property that there exist relatively prime $a, n \in \mbn$ with $C_{E,a,n} = 0$.  Then either there exists a prime $\ell \in \{ 2, 3, 5 \}$ such that $\mbq(E[\ell]) \subseteq \mbq(\zeta_n)$ and $\gs_a |_{\mbq(E[\ell])} \equiv 1$
or there exist $t_0, d_0 \in \mbq$ for which $E$ is isomorphic over $\mbq$ to $\mbe_{t_0,d_0}$, where $\mbe$ is the elliptic curve over $\mbq(t,d)$ defined by \eqref{defofa4anda6} and \eqref{defofEoft}.
\end{Theorem}

\section{Proof of Theorem \ref{mainthm}}
The proof of Theorem \ref{mainthm} falls naturally into three pieces.  We will first prove the statements in part (a) of the theorem, and then we will prove the statements in parts (b) and (c) therein.  At various points in the proof (and later in the paper) we will make use of the symbols
\begin{equation} \label{defofrhosubE}
\begin{split}
&\rho_E : G_\mbq \longrightarrow \GL_2(\hat{\mbz}), \\
&\rho_{E,n} : G_\mbq \longrightarrow \GL_2(\mbz/n\mbz),
\end{split}
\end{equation}
which denote the continuous Galois representations defined by letting $G_\mbq := \gal(\ol{\mbq}/\mbq)$ act on the adelic Tate module $\ds T(E) := \lim_{\leftarrow} E[n]$ (resp. on the $n$-torsion $E[n]$) of $E$, and choosing a $\hat{\mbz}$-basis (resp. a $\mbz/n\mbz$-basis) thereof.  Furthermore, if $G \subseteq \GL_2(\hat{\mbz})$ is any subgroup and $n \in \mbn$, we will denote by $G(n) \subseteq \GL_2(\mbz/n\mbz)$ the image of $G$ under the projection map $\GL_2(\hat{\mbz}) \to \GL_2(\mbz/n\mbz)$.  Note that, with these conventions, if $G := \rho_E(G_\mbq)$ then, for an appropriate choice of basis, we have $G(n) = \rho_{E,n}(G_\mbq)$.  When $d$ divides $n$, we will use the symbol $\pi_{n,d} : \GL_2(\mbz/n\mbz) \to \GL_2(\mbz/d\mbz)$ to denote the canonical projection map. 
\subsection{Preliminaries on modular curves} \label{preliminariesonmodularcurvessection}
Suppose that $\tilde{G} \subseteq \GL_2(\hat{\mbz})$ is an open subgroup (equivalently, of finite index in $\GL_2(\hat{\mbz})$) which additionally satisfies
\[
-I \in \tilde{G} \quad \text{ and } \quad  \det \tilde{G} = \hat{\mbz}^\times.
\]
There is then associated to $\tilde{G}$ a smooth, projective, geometrically irreducible curve $X_{\tilde{G}}$ defined over $\mbq$.  This modular curve $X_{\tilde{G}}$ comes equipped with a forgetful map 
\[
\begin{tikzcd}
j_{\tilde{G}} : X_{\tilde{G}} \rar & X(1) \rar{\simeq} & \mbp^1.
\end{tikzcd}
\]
Furthermore, for every elliptic curve $E$ over $\mbq$ whose $j$-invariant $j_E$ satisfies $j_E \notin \{ 0, 1728 \}$, we have
$\rho_{E}(G_\mbq)$ is $\GL_2(\hat{\mbz})$-conjugate to a subgroup of $\tilde{G}^t := \{ g^t : g \in \tilde{G} \}$ if and only if $j_E \in j_{\tilde{G}}\left( X_{\tilde{G}}(\mbq) \right)$.  
We may also define a generic Weierstrass model over the field $\mbq(d)\left( X_{\tilde{G}} \right)$ by
\[
\mbe^{(\tilde{G})} : y^2 = x^3 + \frac{108d^2j_{\tilde{G}}}{1728 - j_{\tilde{G}}}x + \frac{432d^3j_{\tilde{G}}}{1728 - j_{\tilde{G}}},
\]
whose $j$-invariant is $j_{\tilde{G}}$.  For any $t_0 \in X_{\tilde{G}}(\mbq) - j_{\tilde{G}}^{-1}\left( \{ 0, 1728 \} \right)$ and $d_0 \in \mbq - \{ 0 \}$, we may consider the specialization
\[
\mbe^{(\tilde{G})}_{t_0,d_0} : y^2 = x^3 + \frac{108d_0^2j_{\tilde{G}}(t_0)}{1728 - j_{\tilde{G}}(t_0)}x + \frac{432d_0^3j_{\tilde{G}}(t_0)}{1728 - j_{\tilde{G}}(t_0)},
\]
which is an elliptic curve over $\mbq$, provided $d_0 \neq 0$ and $t_0$ avoids a certain finite subset of $X_{\tilde{G}}(\mbq)$ for which the specialization is not smooth.
For any elliptic curve $E$ over $\mbq$ with $j_E \notin \{ 0, 1728 \}$, we have that $j_E \in j_{\tilde{G}}\left( X_{\tilde{G}}(\mbq) \right)$ if and only if $E$ is isomorphic over $\mbq$ to such a specialization $\mbe^{(\tilde{G})}_{t_0,d_0}$.  Henceforth, let us use the following notation: for subgroups $H_1, H_2 \subseteq \GL_2(\hat{\mbz})$,
\begin{equation} \label{defofdotsubseteq}
H_1 \, \dot\subseteq \, H_2 \; \myeq \; \exists g \in \GL_2(\hat{\mbz}) \text{ such that } gH_1g^{-1} \subseteq H_2.
\end{equation}
For any $m \in \mbn$, we define the relation $H_1 \,\dot\subseteq\, H_2$ for subgroups $H_1, H_2 \subseteq \GL_2(\mbz/m\mbz)$ similarly as subset containment up to conjugation in $\GL_2(\mbz/m\mbz)$.  Summarizing the above,
we have
\begin{equation} \label{modularinterpretation}
\begin{split}
\forall E/\mbq \text{ with } j_E \notin \{ 0, 1728 \}, \; \rho_{E}(G_\mbq) \, \dot\subseteq \, \tilde{G}^t \; &\Longleftrightarrow \; j_E \in j_{\tilde{G}}\left( X_{\tilde{G}}(\mbq) \right) \\
&\Longleftrightarrow \exists t_0 \in X_{\tilde{G}}(\mbq), d_0 \in \mbq \text{ with } E \simeq_\mbq \mbe^{(\tilde{G})}_{t_0,d_0}.
\end{split}
\end{equation}
For full details, we refer the reader to \cite{delignerapoport} (see also \cite{rousedzb} for a helpful discussion about left versus right action of $\GL_2$ on the underlying complete modular curve, which is responsible for the appearance of the transposed group $\tilde{G}^t$ in \eqref{modularinterpretation}).  In case we wish to use modular curves as above to study the question of which elliptic curves $E$ satisfy $\rho_E(G_\mbq) \, \dot\subseteq \, G$ for an open subgroup $G \subseteq \GL_2(\hat{\mbz})$ for which $-I \notin G$, we will always first enlarge $G$ by setting
\begin{equation} \label{defofGtilde}
\tilde{G} := \langle G, -I \rangle \subseteq \GL_2(\hat{\mbz}).
\end{equation}

\subsection{Proof of part (a) of Theorem \ref{mainthm}} \label{proofofpartasection} Let the Borel subgroup $B(2) \subseteq \GL_2(\mbz/2\mbz)$ and the split Cartan subgroup $\mc{C}_s(3) \subseteq \GL_2(\mbz/3\mbz)$ be defined by
\[
\begin{split}
    &B(2) := \left\langle \begin{pmatrix} 1 & 1 \\ 0 & 1 \end{pmatrix} \right\rangle \subseteq \GL_2(\mbz/2\mbz), \\
    &\mc{C}_s(3) := \left\{ \begin{pmatrix} a & 0 \\ 0 & d \end{pmatrix} : a, d \in (\mbz/3\mbz)^\times \right\} \subseteq \GL_2(\mbz/3\mbz).
\end{split}
\]
Let $G(6) \subseteq \GL_2(\mbz/6\mbz)$ be the subgroup that corresponds to $B(2) \times \mc{C}_s(3)$ under the isomoprhism
\[
\GL_2(\mbz/6\mbz) \simeq \GL_2(\mbz/2\mbz) \times \GL_2(\mbz/3\mbz)
\]
of the Chinese remainder theorem and define $G_6 := \pi^{-1}(G(6)) \subseteq \GL_2(\hat{\mbz})$, where $\pi : \GL_2(\hat{\mbz}) \to \GL_2(\mbz/6\mbz)$ is the usual projection map.

Since $-I \in G_6$ and $\det(G_6) = \hat{\mbz}^\times$, we are in the setting of the previous section, and so there is a smooth, projective, geometrically irreducible modular curve $X_{G_6}$ with a forgetful map $j_{G_6} : X_{G_6} \to X(1) \simeq \mbp^1$, and we denote the generic Weierstrass model $\mbe^{(G_6)}$ simply by $\mbe$.  The specializations $\mbe_{t_0,d_0}$ of this Weierstrass model satisfy the property
\begin{equation} \label{specialmodularinterpretation}
\begin{split}
\forall E/\mbq \text{ with } j_E \notin \{ 0, 1728 \}, \; \rho_{E}(G_\mbq) \, \dot\subseteq \, G_6 \; &\Longleftrightarrow \; j_E \in j_{G_6}\left( X_{G_6}(\mbq) \right) \\
&\Longleftrightarrow \exists t_0 \in X_{G_6}(\mbq), d_0 \in \mbq \text{ with } E \simeq_\mbq \mbe_{t_0,d_0}.
\end{split}
\end{equation}
Moreover, a calculation shows that $X_{G_6}$ has genus zero and satisfies $X_{G_6}(\mbq) \neq \emptyset$.  Thus, $X_{G_6} \simeq_\mbq \mbp^1$, and, fixing a parameter $t$ on $X_{G_6}$, the above-mentioned forgetful map may be realized as a rational map $j_{\mbe} : \mbp^1(t) \longrightarrow \mbp^1(j)$ making the diagram
\begin{equation} \label{fancydiagram}
\begin{tikzcd}
X_{G_6} \arrow[black, bend left]{rr}{\text{forgetful map $j_{G_6}$}} \rar{\simeq} & \mbp^1_\mbq(t) \rar{j_{\mbe}} & \mbp^1_\mbq(j)
\end{tikzcd}
\end{equation}
commute.  Our goal is to produce a rational function $j_\mbe(t) \in \mbq(t)$ that represents the above rational map $j_{\mbe}$.

As may be deduced from Theorems 1.1 and 1.2 of \cite{zywina} (see also Tables 1 and 3 of \cite{sutherlandzywina}), for any elliptic curve $E$ over $\mbq$ with $j$-invariant $j_E \notin \{ 0, 1728 \}$, we have that
\begin{equation} \label{conditionsat2and3}
\begin{split}
    [\mbq(E[2]) : \mbq] \leq 2 \; \Longleftrightarrow \; \rho_{E,2}(G_\mbq) \, \dot\subseteq \, B(2) \; &\Longleftrightarrow \; \exists t_0 \in \mbq \text{ with  } j_E = 256 \frac{(t_0+1)^3}{t_0}, \\
    [\mbq(E[3]) : \mbq] \leq 4 \; \Longleftrightarrow \; \rho_{E,3}(G_\mbq) \, \dot\subseteq \, \mc{C}_s(3) \; &\Longleftrightarrow \; \exists t_0 \in \mbq \text{ with  } j_E = 27 \frac{(t_0+1)^3(t_0+3)^3(t_0^2+3)^3}{t_0^3(t_0^2+3t_0+3)^3}.
\end{split}
\end{equation}
Thus, an elliptic curve $E$ over $\mbq$ simultaneously satisfies each of the left-hand conditions if and only $j_E$ is simultaneously a value of each of the rational functions on the right-hand side of \eqref{conditionsat2and3}, and we are led to the algebraic curve defined by 
\begin{equation} \label{algebraiccurve}
X_{G_6} : \; 256 \frac{(u+1)^3}{u} = 27 \frac{(v+1)^3(v+3)^3(v^2+3)^3}{v^3(v^2+3v+3)^3} \quad \text{ (singular model)}.
\end{equation}  
Using {\tt{magma}} \cite{MAGMA} to resolve the singularities, we arrive at the rational functions
\[
v = \frac{4t^3 - 3t - 1}{3t}, \quad u = 256\frac{t^3(t-1)^3(t^2+t+1)^3}{(2t+1)^3(4t^2-2t+1)^3}.
\]
Substituting these into \eqref{algebraiccurve} yields equality, and the resulting rational function $j_{\mb{E}}(t) \in \mbq(t)$ is as in \eqref{defofjofy}.  Viewing $t \in \mbq\left( X_{G_6} \right)$, we have
$
\mbq(X_{G_6}) = \mbq(t)
$,
and $j_{\mb{E}}(t)$ may be taken to define the rational map $j_{\mbe}$ in \eqref{fancydiagram}.  Part (a) of Theorem \ref{mainthm} then follows from this and \eqref{conditionsat2and3}, \eqref{algebraiccurve} and \eqref{specialmodularinterpretation}. \\

\subsection{Proof of part (b) of Theorem \ref{mainthm}}
The proof of part (b) of Theorem \ref{mainthm} will utilize the following lemma, which, for any elliptic curve $E$ over $\mbq$, interprets the acyclicity of $\tilde{E}_p(\mbf_p)$ for a good prime $p$ in terms of $p$ splitting completely in an appropriate division field of $E$.
\begin{lemma} \label{acyclicityofellipticcurvemoduloprime}
Let $E/\mbq$ be an elliptic curve of conductor $N_E$, and $p$ a prime with $p \nmid N_E$. Let $\ell \neq p$ be a prime. Then $\tilde{E}_p(\mbf_p)$ contains a subgroup isomorphic to $\mbz/\ell\mbz \times \mbz/\ell\mbz$ if and only if $p$ splits completely in $\mbq(E[\ell])$.
\end{lemma}
\begin{proof}
See \cite[Lemma 2.1]{cojocarumurty}.
\end{proof}
Now let $\mbe$ be as in \eqref{defofEoft}, let $t_0 \in \mbq - \{ 0, 1, -1/2 \}$ and $d_0 \in \mbq - \{ 0 \}$ satisfy the conditions in part (b) of Theorem \ref{mainthm} and let $\mbe_{t_0,d_0}$ be the specialization of $\mbe$ at $(t_0,d_0)$.  Let $n_0 \in \mbn$ be as in \eqref{defofnsub0} and let $a_0 \in \mbz$ be chosen coprime with $n_0$ and so that \eqref{defofasub0} holds.  We will verify that, for each $x \geq 0$, $\pi_{\mbe_{t_0,d_0},a_0,n_0}(x) \leq 2$.

Consider the field
\[
\mbq(\mathbb{E}_{t_0,d_0}[6]) = \mbq \left( \sqrt{h_2(t_0)}, \sqrt{-3n_3(t_0,d_0)}, \sqrt{-3} \right).
\]
It follows from \eqref{defofasub0} that, for any prime $p \equiv a_0 \bmod n_0$, the automorphism $\gs_p = \gs_{a_0} \in \gal\left( \mbq(\zeta_{n_0})/\mbq \right)$ must either act trivially on $\mbq\left( \sqrt{h_2(t_0)} \right) = \mbq\left( \mbe_{t_0,d_0}[2] \right)$ or on $\mbq\left( \sqrt{-3}, \sqrt{-3n_3(t_0,d_0)} \right) = \mbq\left( \mbe_{t_0,d_0}[3] \right)$.  It then follows from Lemma \ref{acyclicityofellipticcurvemoduloprime} that, for each prime $p \geq 5$ of good reduction for $\mbe_{t_0,d_0}$, the group $\left(\widetilde{\mbe_{t_0,d_0}}\right)_p(\mbf_p)$ either contains a subgroup isomorphic to $\mbz/2\mbz \times \mbz/2\mbz$ or a subgroup isomorphic to $\mbz/3\mbz \times \mbz/3\mbz$.  In particular, for every good prime $p \geq 5$, $\left(\widetilde{\mbe_{t_0,d_0}}\right)_p(\mbf_p)$ is not a cyclic group, and thus $\pi_{\mbe_{t_0,d_0},a_0,n_0}(x) \leq 2$ for every $x \geq 0$, as asserted. \hfill $\Box$ \\

\subsection{Proof of part (c) of Theorem \ref{mainthm}}

To verify part (c) of Theorem \ref{mainthm}, we will make use of the following result, which is \cite[Theorem 1.1, 1.3]{gonzalezjimezezlozanorobledo}
\begin{Theorem} \label{abeliandivisionfield} 
For any elliptic curve $E$ defined over $\mbq$,  $\gal(\mbq(E[d])/\mbq)$ is abelian only if $d \in \{1,2,3,4,5,6,8\}$.
\end{Theorem}
Theorem \ref{abeliandivisionfield} allows us to restrict our verification of $\mbq(\mbe_{t_0,d_0}[\ell]) \not\subseteq \mbq(\zeta_{n_0})$ to just those $\ell \in \{2, 3, 5\}$.  By assumption, the prime $3$ ramifies in $\mbq\left( \sqrt{h_2(t_0)} \right)$, and therefore $3$ does \emph{not} ramify in the field $\mbq\left( \sqrt{-3h_2(t_0)}, \sqrt{-3n_3(t_0,d_0)h_2(t_0)} \right)$, and, by hypothesis, neither does $5$.  Thus, neither $3$ nor $5$ divides $n_0$, and so neither $3$ nor $5$ ramifies in $\mbq(\zeta_{n_0})$.  Since $3$ \emph{does} ramify in $\mbq\left( \sqrt{h_2(t_0)} \right) = \mbq\left( \mbe_{t_0,d_0}[2] \right)$ and in $\mbq\left( \mbe_{t_0,d_0}[3] \right)$, we see that
\[
\mbq\left( \mbe_{t_0,d_0}[2] \right) \not\subseteq \mbq(\zeta_{n_0}) \; \text{ and } \; \mbq\left( \mbe_{t_0,d_0}[3] \right) \not\subseteq \mbq(\zeta_{n_0}).
\]
Similarly, since $5$ ramifies in $\mbq(\mbe_{t_0,d_0}[5])$, we further conclude that $\mbq(\mbe_{t_0,d_0}[5]) \not\subseteq \mbq(\zeta_{n_0})$, finishing the proof of part (c) of Theorem \ref{mainthm}. \hfill $\Box$ \\

\section{A criterion for $\ds C_{E,a,n} = 0$ and a proof of Theorem \ref{secondthm}} \label{Csubansection}

Finally, in this section we give a justification of the condition \eqref{refinedsufficientforfinitelymanyprimes}, which refines \eqref{sufficientforfinitelymanyprimes}, detailing when an elliptic curve $E$ over $\mbq$ and a pair $(a,n) \in \mbn^2$ with $\gcd(a,n) = 1$ should satisfy $\ds \lim_{x \to \infty} \pi_{E,a,n}(x) < \infty$.  This is conjecturally equivalent to the condition that $C_{E,a,n} = 0$; we begin by describing the constant $C_{E,a,n}$ in more detail.  
\subsection{Heuristics connecting $C_{E,a,n}$ with $\pi_{E,a,n}(x)$}
We begin by noting that (thanks to the Hasse bound \mbox{$|p + 1 - |\tilde{E}_p(\mbf_p)| | \leq 2\sqrt{p}$)}, for any prime $p \nmid 2N_E$, $\tilde{E}_p(\mbf_p)$ is cyclic if and only if every prime $\ell \neq p$ satisfies $\mbz/\ell\mbz \times \mbz/\ell\mbz \not\subseteq \tilde{E}_p(\mbf_p)$. 
Thus, by Lemma \ref{acyclicityofellipticcurvemoduloprime}, we have
\begin{equation} \label{biconditionalforcyclicity}
\tilde{E}_p(\mbf_p) \text{ is cyclic} \; \Longleftrightarrow \; \forall \text{ prime }\ell \neq p, \, \frob_p \neq 1 \in \gal\left( \mbq(E[\ell])/\mbq \right),
\end{equation}
where here and henceforth, $\frob_p$ denotes any choice of Frobenius automorphism at $p$ in a given Galois group (which should be clear from context). In particular, we have
\begin{equation} \label{pisubEanfinitecondition}
\lim_{x \to \infty} \pi_{E,a,n}(x) < \infty \; \Longleftrightarrow \; \begin{matrix} \exists S \subseteq \{ \text{primes} \} \text{ with } | S | < \infty \text{ satisfying that } \forall p \notin S \text{ with} \\ p \equiv a \bmod{n}, \, \exists \ell \text{ prime and } \frob_p = 1 \in \gal\left( \mbq(E[\ell])/\mbq \right). \end{matrix}
\end{equation}
The constant $C_{E,a,n}$ encodes the conditional probability of the event \eqref{biconditionalforcyclicity}, given that $p \equiv a \bmod{n}$, as follows.  Grouping the primes $\ell$ on the right-hand side of \eqref{biconditionalforcyclicity} according to whether they divide a ``test level'' $m \in \mbn$, we are led to the biconditional
\begin{equation} \label{mlevelversion}
\tilde{E}_p(\mbf_p) \text{ is cyclic } \; \Longleftrightarrow \; \forall m \in \mbn \text{ with } p \nmid m \text{ and } \forall \text{ prime }\ell \mid m, \, \frob_p \neq 1 \in \gal\left( \mbq(E[\ell])/\mbq \right).
\end{equation}
To study the density of such primes while conditioning on $p \equiv a \bmod{n}$ (equivalently, conditioning on $\frob_p = \gs_a \in \gal\left( \mbq(\zeta_n)/\mbq \right)$), we define the following sets:
\begin{equation*}
\begin{split}
    S_{E,a,n}(m) &:= \left\{ \gs \in \gal\left( \mbq(E[m])\mbq(\zeta_n)/\mbq \right) : \gs |_{\mbq(\zeta_n)} = \gs_a \right\}, \\
    S_{E,a,n}'(m) &:= \left\{ \gs \in S_{E,a,n}(m) : \forall \ell \mid m, \; \gs |_{\mbq(E[\ell])} \neq 1 \right\}, \\
    S_{E,a,n}^{(d)}(m) &:= \left\{ \gs \in S_{E,a,n}(m) : \gs |_{\mbq(E[d])} = 1 \right\}.
\end{split}
\end{equation*}
Note that $S_{E,a,n}^{(1)}(m) = S_{E,a,n}(m)$.  The ``probability visible at level $m$'' that $\tilde{E}_p(\mbf_p)$ is cyclic is reflected in the right-hand condition in \eqref{mlevelversion} for a \emph{fixed} value of $m$, and is given by
\begin{equation} \label{probabilityforfixedm}
\prob_m\left( \tilde{E}_p(\mbf_p) \text{ is cyclic, given that } p \equiv a \bmod{n} \right) = \frac{ \left| S_{E,a,n}'(m) \right| }{\left|\gal\left(\mbq(E[m])\mbq(\zeta_n)/\mbq\right) \right|}.
\end{equation}
For any $d$ dividing $m$, we let \mbox{$\varpi_{m,d} : \gal\left( \mbq(E[m])\mbq(\zeta_n)/\mbq \right) \to \gal\left( \mbq(E[d])\mbq(\zeta_n)/\mbq \right)$} denote the restriction map.  Using the observation
$
| S_{E,a,n}^{(d)}(d) | = \gamma_{a,n}\left( \mbq(E[d]) \right)
$
(see \eqref{defofgammasuban}), we see that
\[
\frac{ | S_{E,a,n}^{(d)}(m) | }{|\gal\left(\mbq(E[m])\mbq(\zeta_n)/\mbq\right) |}
=
\frac{ \left| \varpi_{m,d}^{-1}\left(S_{E,a,n}^{(d)}(d) \right) \right| }{| \varpi_{m,d}^{-1}\left(\gal\left(\mbq(E[d])\mbq(\zeta_n)/\mbq\right) \right) |}
=
\frac{\gamma_{a,n}\left( \mbq(E[d]) \right)}{\left| \gal\left(\mbq(E[d])\mbq(\zeta_n)/\mbq\right) \right|}.
\]
Moreover, since $\ds S_{E,a,n}'(m) = S_{E,a,n}^{(1)}(m) - \bigcup_{\ell \mid m} S_{E,a,n}^{(\ell)}(m)$, we may apply inclusion-exclusion, concluding that
\begin{equation} \label{chebotarevinterpretation}
\frac{ \left| S_{E,a,n}'(m) \right| }{\left|\gal\left(\mbq(E[m])\mbq(\zeta_n)/\mbq\right) \right|} 
= \sum_{d \mid m} \frac{\mu(d)| S_{E,a,n}^{(d)}(m)|}{\left|\gal\left(\mbq(E[m])\mbq(\zeta_n)/\mbq\right) \right|}
= \sum_{d \mid m} \frac{\mu(d) \gamma_{a,n}\left( \mbq(E[d]) \right)}{\left| \gal\left(\mbq(E[d])\mbq(\zeta_n)/\mbq\right) \right|}.
\end{equation}
Thus, taking the limit in \eqref{probabilityforfixedm} as $m \to \infty$ through any sequence that is cofinal with respect to divisibility (for instance we may simply take $m_n := \ds \prod_{\ell \leq n} \ell^n$) we arrive at the heuristic density
\begin{equation} \label{CsubEanprime}
C_{E,a,n} = \lim_{n \to \infty} \sum_{d \mid m_n} \frac{\mu(d) \gamma_{a,n}\left( \mbq(E[d]) \right)}{\left| \gal\left(\mbq(E[d])\mbq(\zeta_n)/\mbq\right) \right|} = \sum_{d = 1}^\infty \frac{\mu(d) \gamma_{a,n}\left( \mbq(E[d]) \right)}{[\mbq(E[d])\mbq(\zeta_n) : \mbq ]}.
\end{equation}
\begin{remark} \label{SEanofmremark}
Since $\gal(\mbq(E[m])\mbq(\zeta_n)/\mbq)$ is isomorphic to the group
\[
\{ (g_m,g_n) \in \gal(\mbq(E[m])/\mbq) \times \gal(\mbq(\zeta_n)/\mbq) : g_m |_{\mbq(E[m]) \cap \mbq(\zeta_n)} = g_n |_{\mbq(E[m]) \cap \mbq(\zeta_n)} \},
\]
we may see that $S_{E,a,n}(m)$ is in one-to-one correspondence with the set
\[
\{ \gs \in \gal(\mbq(E[m])/\mbq) : \gs |_{\mbq(E[m]) \cap \mbq(\zeta_n)} = \gs_a |_{\mbq(E[m]) \cap \mbq(\zeta_n)} \},
\]
and similarly with the sets $S_{E,a,n}'(m)$ and $S_{E,a,n}^{(d)}(m)$.  We will make use of this later.
\end{remark}
\subsection{The constant $C_{E,a,n}$ as an ``almost Euler product.''}
We will presently describe the constant $C_{E,a,n}$ as a convergent Euler product multiplied by a certain rational number, which will allow us to verify (and refine) the biconditional statement \eqref{refinedsufficientforfinitelymanyprimes}.
In particular, we will be able to add the condition that any prime $\ell$ occurring on the right hand-side of \eqref{pisubEanfinitecondition} be restricted to lie in a finite set.  Toward this end, let us describe in some detail the nature of the image of the Galois representation $\rho_E$ in \eqref{defofrhosubE}.  In case $E$ has CM by the order $\mc{O}_{K,f} \subseteq \mc{O}_K$ of conductor $f \in \mbn$ inside an imaginary quadratic field $K$, the image of $\rho_E$ is not open inside the profinite group $\GL_2(\hat{\mbz})$, but it is open inside a particular subgroup, which we now specify, following \cite{lozanorobledoCM}. Let $\gD_K \in \mbz$ denote the discriminant of $K$ and define the integers $\gd = \gd_{K,f}$ and $\phi = \phi_{K,f}$ by
\[
(\gd,\phi) :=
\begin{cases} 
\left( \frac{\gD_K f^2}{4}, 0 \right) & \text{ if } \gD_K f^2 \equiv 0 \bmod{4}, \\
\left( \frac{(\gD_K -1)}{4} f^2, f \right) & \text{ if } \gD_K f^2 \equiv 1 \bmod{4}.
\end{cases}
\]
Then define the subgroups $\mc{C}_{\gd,\phi}(n) \subseteq  \mc{N}_{\gd,\phi}(n) \subseteq \GL_2(\mbz/n\mbz)$ by
\[
\begin{split}
\mc{C}_{\gd,\phi}(\mbz/n\mbz) &:= 
\left\{ \begin{pmatrix} a + b\phi & b \\ \gd b & a \end{pmatrix} : a, b \in \mbz/n\mbz, a^2 + \phi ab - \gd b^2 \in (\mbz/n\mbz)^\times \right\}, \\
\mc{N}_{\gd,\phi}(\mbz/n\mbz) &:= \left\langle \mc{C}_{\gd,\phi}(\mbz/n\mbz), \begin{pmatrix} -1 & 0 \\ \phi & 1 \end{pmatrix} \right\rangle.
\end{split}
\]
Finally, set $\mc{N}_{\gd,\phi}(\hat{\mbz}) := \ds \lim_{\leftarrow} \mc{N}_{\gd,\phi}(\mbz/n\mbz)$.  If $E$ has CM by the imaginary quadratic order $\mc{O}_{K,f}$ then, for an appropriate choice of basis, we have 
\begin{equation} \label{CMcontainment}
\rho_E(G_\mbq) \subseteq \mc{N}_{\gd,\phi}(\hat{\mbz}).
\end{equation}
(Henceforth we will assume that, in the CM case, the underlying choice of basis is made so that \eqref{CMcontainment} holds.) To uniformize notation, let us define $\mathbb{G}_E(\mbz/n\mbz)$ by
\begin{equation*}
    \mathbb{G}_E(\mbz/n\mbz) := 
    \begin{cases}
        \GL_2(\mbz/n\mbz) & \text{ if $E$ has no CM,} \\
        \mc{N}_{\gd,\phi}(\mbz/n\mbz) & \text{ if $E$ has CM by $\mc{O}_{K,f}$,}
    \end{cases}
\end{equation*}
and $\mathbb{G}_E(\hat{\mbz}) := \ds \lim_{\leftarrow} \mathbb{G}_E(\mbz/n\mbz)$.
Thanks to Serre's open image theorem in the non-CM case and class field theory in the CM case, we have
$    [ \mathbb{G}_E(\hat{\mbz}) : \rho_E(G_\mbq) ] < \infty$.  
It follows that there exists $m \in \mbn$ for which \begin{equation} \label{defofmsubE}
    \ker\left( \mathbb{G}_E(\hat{\mbz}) \to \mathbb{G}_E(\mbz/m\mbz)\right) \subseteq \rho_E\left( G_\mbq \right).
\end{equation}
We define $m_E \in \mbn$ to be the smallest $m \in \mbn$ for which \eqref{defofmsubE} holds, and call it the \emph{$\GL_2$-level} of the group $\rho_E(G_\mbq)$.

We are now ready to analyze the constant $C_{E,a,n}$ in further detail.  Suppose that $f : \mbn \to \mbc$ is any function for which there exists $M \in \mbn$ so that
\begin{equation} \label{almostmultiplicative}
\forall d_1, d_2 \in \mbn, \; \gcd(Md_1,d_2) = 1 \; \Longrightarrow \; f(d_1d_2) = f(d_1)f(d_2),
\end{equation}
and for which the infinite sum
$
\ds \sum_{d = 1}^\infty \mu(d) f(d)
$
converges absolutely.  Writing any $d \in \mbn$ as $d = d_1 d_2$ with $d_1$ only divisible by primes dividing $M$ and $\gcd(d_2,M) = 1$, it follows from \eqref{almostmultiplicative} that
\begin{equation} \label{wheretostickf}
\begin{split}
\sum_{d = 1}^\infty \mu(d) f(d) &= \left( \sum_{d_1 \mid M} \mu(d_1) f(d_1) \right) \left( \sum_{{\begin{substack} {d_2 \in \mbn \\ \gcd(d_2,M) = 1} \end{substack}}} \mu(d_2) f(d_2) \right) \\
&= \left( \sum_{d_1 \mid M} \mu(d_1) f(d_1) \right) \prod_{{\begin{substack} {\ell \text{ prime} \\ \ell \nmid M} \end{substack}}} \left( 1 - f(\ell) \right).
\end{split}
\end{equation}
We will apply the above with $\ds f(d) := \frac{\gamma_{a,n}\left( \mbq(E[d]) \right)}{\left[ \mbq(E[d]) \mbq(\zeta_n) : \mbq(\zeta_n) \right]}$ and $M = m_E$.  This application is justified by our next lemma.
\begin{lemma} \label{almostmultiplicativelemma}
Define the arithmetic function $f : \mbn \longrightarrow \mbr$ by $\ds f(d) := \frac{\gamma_{a,n}\left( \mbq(E[d]) \right)}{\left[ \mbq(E[d]) \mbq(\zeta_n) : \mbq(\zeta_n) \right]}$.  For any $d_1, d_2 \in \mbn$, we have
\begin{equation} \label{quasimultiplicative}
\gcd(d_1 m_E, d_2) = 1 \; \Longrightarrow \; f(d_1d_2) = f(d_1) f(d_2),
\end{equation}
where $m_E$ is defined as above to be the smallest $m \in \mbn$ for which \eqref{defofmsubE} holds.
\end{lemma}
\begin{proof}
It suffices to prove \eqref{quasimultiplicative} with $f$ replaced by its numerator (resp. by its denominator).  For any $d, d' \in \mbn$, with $d$ dividing $d'$, we have
\[
\gs_a |_{\mbq(E[d]) \cap \mbq(\zeta_n)} \neq 1 \; \Longrightarrow \; \gs_a |_{\mbq(E[d']) \cap \mbq(\zeta_n)} \neq 1,
\]
and thus $\gamma_{a,n}\left( \mbq(E[d]) \right) = 0$ implies that $\gamma_{a,n}\left( \mbq(E[d']) \right) = 0$.  Thus, the only way \eqref{quasimultiplicative} can fail with $\gamma_{a,n}\left( \mbq(E[d]) \right)$ in place of $f(d)$ is in case $\gamma_{a,n}\left( \mbq(E[d_1]) \right) = \gamma_{a,n}\left( \mbq(E[d_2]) \right) = 1$ but $\gamma_{a,n}\left( \mbq(E[d_1d_2]) \right) = 0$.  If this is the case, it follows that
\begin{equation} \label{todisprove}
\left( \mbq(E[d_1]) \cap \mbq(\zeta_n) \right) \cdot \left( \mbq(E[d_2]) \cap \mbq(\zeta_n) \right) \subsetneq \left( \mbq(E[d_1d_2]) \cap \mbq(\zeta_n) \right).
\end{equation}
We will presently show that this strict containment cannot happen, but first let us consider in tandem the denominator of $f$.  In what follows, we set
\[
K := \mbq(\zeta_n).
\]
By Galois theory, we have
\[
\left[ K(E[d_1d_2]) : K \right] \neq \left[ K(E[d_1]) : K \right] \left[ K(E[d_2]) : K \right] \; \Leftrightarrow \; K(E[d_1]) \cap K(E[d_2]) \neq K.
\]
Assuming that $\gcd(d_1 m_E, d_2) = 1$, a Galois theoretic analysis of an appropriate field diagram further implies that 
\[
\left[ \mbq(E[d_1d_2]) \cap K : \left( \mbq(E[d_1]) \cap K \right) \cdot \left( \mbq(E[d_2]) \cap K \right) \right] 
=
\left[ K( E[d_1] ) \cap K( E[d_2] ) : K \right].
\]
Therefore, if \eqref{quasimultiplicative} fails then \eqref{todisprove} must hold; our goal is now to prove that \eqref{todisprove} cannot hold if $\gcd(d_1m_E,d_2) = 1$.  

Assume for the sake of contradiction that \eqref{todisprove} holds for some $d_1, d_2 \in \mbn$ with $\gcd(d_1m_E,d_2) = 1$.  Replacing each of $d_1$, $d_2$ and $n$ with a proper divisor as necessary, we may and will assume that, for any proper divisor $d_2'$ of $d_2$, \eqref{todisprove} is false (i.e. the strict inclusion becomes an equality) when $d_2$ is replaced with $d_2'$ and likewise with any proper divisor $d_1'$ of $d_1$ and any proper divisor $n'$ of $n$.  In particular, this together with \eqref{todisprove} implies that
\begin{equation} \label{conductor}
n = \cond\left( \mbq(E[d_1d_2]) \cap \mbq(\zeta_n) \right).
\end{equation}
We now write $N := \lcm(n,d_1d_2)$ and consider the image $\rho_{E,N}(G_\mbq) =: G(N) \subseteq \GL_2(\mbz/N\mbz)$.  Write $N = N_1N_2$ (resp. $n = n_1n_2$) with $N_1$ (resp. $n_1$) supported on primes dividing $m_E$ and $\gcd(N_2,m_E) = 1$ (resp. $\gcd(n_2,m_E) = 1$).
We then have $d_2 \mid N_2$ and, since $\gcd(N_1m_E,N_2) = 1$, under the isomorphism of the Chinese remainder theorem,
\begin{equation} \label{fulldirectproduct}
G(N_1N_2) \simeq G(N_1) \times \GL_2(\mbz/N_2\mbz) \subseteq \pi_{N_1,d_1}^{-1}\left( G(d_1) \right) \times \GL_2(\mbz/N_2\mbz).
\end{equation}

Let us set $A := \gal\left( \mbq(E[d_1d_2]) \cap \mbq(\zeta_n) / \mbq \right)$ and define the surjective group homomorphisms
\[
\psi, \chi : \pi_{N_1,d_1}^{-1}\left( G(d_1) \right) \times \GL_2(\mbz/N_2\mbz) \twoheadrightarrow A,
\]
by declaring $\psi$ to be the reduction map modulo $d_1d_2$ followed by the map which corresponds under $G(d_1) \times \GL_2(\mbz/d_2\mbz) \simeq G(d_1d_2) \simeq \gal\left( \mbq(E[d_1d_2])/\mbq \right)$ to the restriction map and setting $\chi$ be the determinant map modulo $n$ followed by the map which corresponds under $(\mbz/n\mbz)^\times \simeq \gal(\mbq(\zeta_n)/\mbq)$ to the restriction map onto $\gal(\mbq(E[d_1d_2]) \cap \mbq(\zeta_n)/\mbq)$.  

We next define $\psi_1, \chi_1 : \pi_{N_1,d_1}^{-1}\left( G(d_1) \right) \to A$ by $\psi_1(g_1) := \psi(g_1,1)$ and $\chi_1 (g_1) := \chi ( g_1,1)$ and $\psi_2, \chi_2  : \GL_2(\mbz/N_2\mbz) \to A$ by $\psi_2(g_2) := \psi(1,g_2)$ and $\chi_2 (g_2) := \chi (1,g_2)$.  We then clearly have $\psi(g_1,g_2) = \psi_1(g_1)\psi_2(g_2)$ and $\chi (g_1,g_2) = \chi_1(g_1)\chi_2(g_2)$, and
\begin{equation} \label{fiberedproductform}
G(N_1N_2) \subseteq \{ (g_1,g_2) \in \pi_{N_1,d_1}^{-1}\left( G(d_1) \right) \times \GL_2(\mbz/N_2\mbz) : \psi_1(g_1)\chi_1^{-1}(g_1) = \psi_2^{-1}(g_2) \chi_2(g_2) \}.
\end{equation}
We will now show that this contradicts \eqref{fulldirectproduct}. \\

\noindent \textbf{Case 1: $\chi_2 \neq \psi_2$.}  In this case, the right-hand side of \eqref{fiberedproductform} is either a non-trivial fibered product, or else $G(N_2) \subsetneq \GL_2(\mbz/N_2\mbz)$, contradicting \eqref{fulldirectproduct} either way.  Thus, we arrive at a contradiction in this case. \\

\noindent \textbf{Case 2: $\chi_2 = \psi_2$.}  In this case, by \eqref{conductor}, we have $\cond \chi_2 = n_2$, and since $\psi_2$ factors through projection modulo $d_2$, it follows that $n_2$ divides $d_2$, and so $N_2 = d_2$.  Furthermore, by \eqref{fiberedproductform} we have
$
(g_1,g_2) \in G(N_1N_2) \; \Rightarrow \; \psi_1(g_1) = \chi_1(g_1),
$
and so 
\begin{equation} \label{keylevelN1containments}
\mbq(\zeta_{n_1}) \supseteq \mbq(E[N_1])^{\ker \chi_1} = \mbq(E[N_1])^{\ker \psi_1} \subseteq \mbq(E[d_1]),
\end{equation}
the last subset containment implied by $\ker \psi_1 \supseteq \ker \pi_{N_1,d_1}$.  We now extend $\psi_1$ to all of $\pi_{N_1,d_1}^{-1}\left( G(d_1) \right) \times \GL_2(\mbz/N_2\mbz)$ via $\tilde{\psi}_1\left( (g_1,g_2) \right) := \psi_1(g_1)$, so that, by \eqref{keylevelN1containments}, 
\begin{equation} \label{keylevelNcontainment}
\mbq(E[N])^{\ker \tilde{\psi}_1} = \mbq(E[N_1])^{\ker \psi_1} \subseteq \mbq(E[d_1]) \cap \mbq(\zeta_n).
\end{equation}
Considering also the map $\det \circ \pi_{N,d_2} : \pi_{N_1,d_1}^{-1}\left( G(d_1) \right) \times \GL_2(\mbz/N_2\mbz) \twoheadrightarrow (\mbz/d_2\mbz)^\times$, i.e. the determinant map modulo $d_2$, we will now observe that
\begin{equation} \label{kernelcontainment}
\ker \psi \supseteq \ker \tilde{\psi}_1 \cap \ker \left( \det \circ \pi_{N,d_2} \right).
\end{equation}
Indeed, if $(g_1,g_2) \in \ker \tilde{\psi}_1 \cap \ker \left( \det \circ \pi_{N,d_2} \right)$, then $\det g_2 \equiv 1 \bmod{d_2}$, and so $\chi_2(g_2) = 1 = \psi_2(g_2)$.  Therefore $\psi\left( (g_1,g_2) \right) = \psi_1(g_1) \psi_2(g_2) = \tilde{\psi}_1((g_1,g_2)) \psi_2(g_2) = 1$, establishing \eqref{kernelcontainment}.  By the Galois correspondence together with \eqref{keylevelNcontainment},  this implies that
\[
\begin{split}
\mbq(E[d_1d_2]) \cap \mbq(\zeta_n) = \mbq(E[N])^{\ker \psi} &\subseteq \mbq(E[N])^{\ker \tilde{\psi}_1} \cdot \mbq(E[N])^{\ker \left( \det \circ \pi_{N,d_2} \right)} \\
&\subseteq \left( \mbq(E[d_1]) \cap \mbq(\zeta_n) \right) \cdot \mbq(\zeta_{d_2}) \\
&\subseteq \left( \mbq(E[d_1]) \cap \mbq(\zeta_n) \right) \cdot \left( \mbq(E[d_2]) \cap \mbq(\zeta_{n}) \right),
\end{split}
\]
contradicting \eqref{todisprove}.  Thus, we have also arrived at a contradiction in this second case.  This establishes in all cases that \eqref{todisprove} is inconsistent with $\gcd(d_1m_E,d_2) = 1$, as desired.
\end{proof}
Returning to \eqref{wheretostickf}, we set $\ds f(d) := \frac{\gamma_{a,n}\left( \mbq(E[d]) \right)}{\left[ \mbq(E[d]) \mbq(\zeta_n) : \mbq(\zeta_n) \right]}$ and $M = m_E$.  By Lemma \ref{almostmultiplicativelemma} together with \eqref{CsubEanprime} and \eqref{chebotarevinterpretation}, we conclude that
\[
C_{E,a,n} =  \frac{|S_{E,a,n}'\left( \rad(m_E) \right)|}{|\gal\left( \mbq(E[\rad(m_E)])\mbq(\zeta_n)/ \mbq\right) |} \cdot \prod_{{\begin{substack} {\ell \text{ prime} \\ \ell \nmid m_E \\ \ell \mid \gcd(n,a-1)} \end{substack}}} \left( 1 -  \frac{ \phi(\ell) }{| \mathbb{G}_E(\mbz/\ell\mbz) |} \right) \cdot \prod_{{\begin{substack} {\ell \text{ prime} \\ \ell \nmid nm_E} \end{substack}}} \left( 1 -  \frac{ 1 }{| \mathbb{G}_E(\mbz/\ell\mbz) |} \right),
\]
where $\ds \rad(m_E) := \prod_{\ell \mid m_E} \ell$.  Since the infinite product converges, it follows that
\[
\begin{split}
C_{E,a,n} = 0 \; &\Longleftrightarrow \; S_{E,a,n}'\left( \rad(m_E) \right) = \emptyset \\
&\Longleftrightarrow \begin{pmatrix} \exists d \mid \rad(m_E) \text{ such that } \forall \gs \in \gal\left(\mbq(E[d])\mbq(\zeta_n)/\mbq\right) \text{ with} \\ \gs |_{\mbq(\zeta_n)} = \gs_a, \, \exists \text{ a prime } \ell \mid d \text{ for which }  \gs |_{\mbq(E[\ell])} = 1 \end{pmatrix}.
\end{split}
\]
In particular, conjecturally, any prime $\ell$ appearing on the right-hand side of 
\eqref{pisubEanfinitecondition} may be assumed to divide $\rad(m_E)$, and \eqref{refinedsufficientforfinitelymanyprimes} may be stated in the refined form
\begin{equation*} 
\begin{pmatrix} \exists d \mid \rad(m_E) \text{ such that } \forall \gs \in \gal\left(\mbq(E[d])\mbq(\zeta_n)/\mbq\right) \text{ with} \\ \gs |_{\mbq(\zeta_n)} = \gs_a, \, \exists \text{ a prime } \ell \mid d \text{ for which }  \gs |_{\mbq(E[\ell])} = 1 \end{pmatrix} \; \Longleftrightarrow \; \lim_{x \to \infty} \pi_{E,a,n}(x) < \infty.
\end{equation*}
\subsection{Proof of Theorem \ref{secondthm}}
To prove Theorem \ref{secondthm}, we will first make a definition and then prove a proposition that reduces the proof to a computation.

\begin{Definition} \label{acyclicitydef}
We call a subgroup $G \subseteq \GL_2(\hat{\mbz})$ an \emph{acyclicity group} if there is a triple $(E,a,n)$ where $E$ is an elliptic curve over $\mbq$ with  $\rho_E(G_\mbq) \, \dot = \, G$ and $a,n \in \mbn$ are relatively prime and satisfy $C_{E,a,n} = 0$.  We will call such a triple $(E,a,n)$ an \emph{acyclicity witness for $G$}.  We define the \emph{acyclicity level} of $G$ to be the smallest $d \in \mbn$ for which 
\begin{equation} \label{emptycondition}
S_{E,a,n}'(d) = \emptyset,
\end{equation}
as $(E,a,n)$ ranges over all acyclicity witnesses of the group $G$.  When $d$ is the acyclicity level of $G$, we call any acyclicity witness $(E,a,n)$ of $G$ that satisfies \eqref{emptycondition}, and for which $n$ is minimal with respect to this condition, a \emph{minimal acyclicity witness for $G$}. 
\end{Definition}
\begin{remark} \label{acyclicityremark}
Let $E$ be an elliptic curve over $\mbq$.  Directly from Definition \ref{acyclicitydef}, we have
\[
\exists \, (a,n) \in \mbn^2 \text{ relatively prime with } C_{E,a,n} = 0 \; \Longleftrightarrow \; \rho_E(G_\mbq) \text{ is an acyclicity group.}
\]
Furthermore, when this is the case, the acyclicity level of $\rho_E(G_\mbq)$ is the smallest $d \in \mbn$ for which the left-hand condition in \eqref{refinedsufficientforfinitelymanyprimes} holds for some (appropriately chosen) $a \in \mbz$ and $n \in \mbn$.  In particular,
\begin{equation*} 
\begin{pmatrix} \exists \text{ a prime } \ell \text{ and coprime } a,n \in \mbn \text{ such} \\
\text{that } \mbq(E[\ell]) \subseteq \mbq(\zeta_n) \text{ and } \gs_a |_{\mbq(E[\ell])} = 1
\end{pmatrix}
\; \Longleftrightarrow \;
\begin{pmatrix} \rho_E(G_\mbq) \text{ is an acyclicity group} \\ \text{with \emph{prime} acyclicity level}
\end{pmatrix}
\end{equation*}
whereas
\begin{equation*} 
\begin{pmatrix} 
\exists  d \in \mbn_{\geq 2} \text{ and } \exists \text{ coprime } a,n \in \mbn \text{ such that} \\ \forall \gs \in \gal\left(\mbq(E[d])\mbq(\zeta_n)/\mbq\right) \text{ with } \gs |_{\mbq(\zeta_n)} = \gs_a, \\ \exists \text{ a prime } \ell \mid d \text{ for which }  \gs |_{\mbq(E[\ell])} = 1, \\
\text{but this condition fails whenever $d$ is prime}
\end{pmatrix}
\; \Longleftrightarrow \;
\begin{pmatrix} \rho_E(G_\mbq) \text{ is an acyclicity group} \\ \text{with \emph{composite} acyclicity level}
\end{pmatrix}.
\end{equation*}
Finally, since the constant $C_{E,a,n}$ only depends on the fields $\mbq(E[d])$ for \emph{square-free} values of $d$, it follows that the acyclicity level of an acyclicity group is necessarily square-free.
\end{remark}
Since any acyclicity witness $(E,a,n)$ satisfies $\rho_E(G_\mbq) \, \dot = \,  G$, the property of being an acyclicity group is inherent in the group $G$.  Proposition \ref{acyclicitylemma} below
lays out some conditions that $G$ must satisfy in order to be an acyclicity group with composite acyclicity level.  First, we state a key group-theoretical lemma, whose proof is straightforward.
\begin{lemma} \label{purelygrouptheorylemma}
Let $G$ be a group and let $N_1, N_2 \, \unlhd \, G$ be normal subgroups of $G$.  Suppose that $gN_1, g'N_1 \in G/N_1$ satisfy $gN_1N_2 = g'N_1N_2$.  Then there exist $n \in N_1$ for which $gnN_2 = g'N_2$. 
\end{lemma}
\begin{cor} \label{keyacyclicitycor}
Let $E$ be an elliptic curve over $\mbq$ and let $d, d', n \in \mbn$ with $d'$ a divisor of $d$.  If $\gs, \gs' \in \gal(\mbq(E[d'])/\mbq)$ satisfy $\gs |_{\mbq(E[d']) \cap \mbq(\zeta_n)} = \gs' |_{\mbq(E[d']) \cap \mbq(\zeta_n)}$, then there exist lifts $\tilde{\gs}, \tilde{\gs}' \in \gal(\mbq(E[d])/\mbq)$ of $\gs, \gs'$ respectively which satisfy $\tilde{\gs} |_{\mbq(E[d]) \cap \mbq(\zeta_n)} = \tilde{\gs}' |_{\mbq(E[d]) \cap \mbq(\zeta_n)}$.
\end{cor}
\begin{proof}
We apply Lemma \ref{purelygrouptheorylemma} with $G := \gal(\mbq(E[d])/\mbq)$, $N_1 := \ker \pi_1$ and $N_2 := \ker \pi_2$, where
\[
\begin{split}
    & \pi_1 : \gal(\mbq(E[d])/\mbq) \longrightarrow \gal(\mbq(E[d'])/\mbq), \\
    & \pi_2 : \gal(\mbq(E[d])/\mbq) \longrightarrow \gal(\mbq(E[d] \cap \mbq(\zeta_n))/\mbq) \\
\end{split}
\]
are the restriction maps.
\end{proof}
\begin{proposition} \label{acyclicitylemma}
Let $G \subseteq \GL_2(\hat{\mbz})$ be an acyclicity group with \textbf{composite} acyclicity level $d$. Then there exists a surjective group homomorphism $\chi : G(d) \twoheadrightarrow A$ onto an abelian group $A$ whose kernel $N(d) := \ker \chi$ satisfies the following three properties:
\begin{enumerate}
\item Defining $\tilde{N}(d) := \prod_{\ell \mid d} N(\ell)$ and viewing $N(d) \subseteq \tilde{N}(d)$ via the Chinese remainder theorem,
there exists an element $\tau \in \tilde{N}(d) \cap G(d)$ such that, for each $\gs \in \tau N(d) \subseteq G(d)$,
there exists a prime $\ell \mid d$ for which $\gs \equiv I \bmod{\ell}$, and for each $\ell \mid d$
there exists $\gs' \in \tau N(d)$ such that, for each prime $\ell'$ dividing $d/\ell$, $\gs' \not\equiv I \bmod{\ell'}$.
\item For each prime $\ell$ dividing $d$, we have $| N(\ell) | > 1$.
\item For any prime $\ell$ dividing $d$, writing $d = \ell \cdot d'$ (with $\ell \nmid d'$), we have $N(d) \cap \ker \pi_{d,d'} = \{ 1 \}$.  In particular, the group $N(\ell)$ is isomorphic to a quotient of the group $N(d')$.
\end{enumerate}
\end{proposition}
\begin{proof}
Let $(E,a,n)$ be a minimal acyclicity witness for $G$; define $A := \gal\left( \mbq(E[d]) \cap \mbq(\zeta_n) / \mbq \right)$ and let {\mbox{$\chi : G(d) \twoheadrightarrow A$}} be the surjective homomorphism corresponding under $G(d) \simeq \gal\left( \mbq(E[d])/\mbq \right)$ to the restriction map $\gal\left( \mbq(E[d])/\mbq \right) \to \gal\left( \mbq(E[d]) \cap \mbq(\zeta_n) / \mbq \right)$.  To prove item (1), pick any $\tau \in \chi^{-1}(\gs_a)$ and note that then $\chi^{-1}(\gs_a) = \tau N(d)$.  Fixing any prime $\ell \mid d$, we observe that $\tau \bmod{\ell} \in N(\ell)$, or else for each $\gs \in \tau N(d)$, $\gs \not\equiv I \bmod{\ell}$, implying that the acyclicity level of $G$ must divide $d/\ell$, a contradiction.  Thus $\tau \in \tilde{N}(d)$; the final two stated properties in item (1) are equivalent to the fact that $d$ is the acyclicity level of $G$. Indeed, the equivalence
\[
S_{E,a,n}'(d) = \emptyset \; \Longleftrightarrow \; \forall \gs \in \tau N(d) \; \exists \ell \mid d \text{ for which } \gs \equiv I \bmod{\ell}
\]
follows directly from the definitions (through the lens of Remark \ref{SEanofmremark}), the identification $G(d) \simeq \gal(\mbq(E[d])/\mbq)$ and the other identifications mentioned above.  For the second of the final two conditions, let $\ell$ be a prime divisor of $d$ and let $d' := d/\ell$.  Suppose for the sake of contradiction that
\begin{equation} \label{forthesakeofcontradiction}
    \forall \gs' \in \tau N(d) \; \exists \text{ a prime } \ell' \mid d' \text{ for which } \gs' \equiv I \bmod{\ell'}.
\end{equation}
We then claim that $S_{E,a,n}'(d') = \emptyset$.  To see this, let $\gs \in \gal(\mbq(E[d'])/\mbq)$ be any automorphism satisfying $\gs |_{\mbq(E[d']) \cap \mbq(\zeta_n)} = \gs_a |_{\mbq(E[d']) \cap \mbq(\zeta_n)}$.  By Corollary \ref{keyacyclicitycor}, we may find a lift $\tilde{\gs} \in \gal(\mbq(E[d])/\mbq)$ of $\gs$ satisfying $\tilde{\gs} |_{\mbq(E[d]) \cap \mbq(\zeta_n)} = \gs_a |_{\mbq(E[d]) \cap \mbq(\zeta_n)}$.  By \eqref{forthesakeofcontradiction}, there exists a prime $\ell' \mid d'$ for which $\tilde{\gs} |_{\mbq(E[\ell'])} = 1$, and since $\tilde{\gs} |_{\mbq(E[\ell'])} = \gs |_{\mbq(E[\ell'])}$, we may see that $S_{E,a,n}'(d') = \emptyset$, contradicting the assumption that $d$ is the acyclicity level of $G$.  This establishes item (1).

To prove item (2), assume for the sake of contradiction that $N(\ell) = \{ 1 \}$ for some prime $\ell$ dividing $d$.  Then $N(d) \subseteq \ker \pi_\ell$, and thus
\[
\mbq(\zeta_n) \supseteq \mbq(E[d]) \cap \mbq(\zeta_n) = \mbq(E[d])^{\ker \chi} \supseteq \mbq(E[d])^{\ker \pi_\ell} = \mbq(E[\ell]).
\]
Consider now the restriction to $\mbq(E[\ell])$ of $\gs_a \in \gal(\mbq(\zeta_n)/\mbq)$.  If $\gs_a |_{\mbq(E[\ell])} = 1$, then $S'_{E,a,n}(\ell) = \emptyset$, contradicting that the acyclicity level of $G$ is composite.  If on the other hand 
\begin{equation} \label{gsaisnotidentity}
\gs_a |_{\mbq(E[\ell])} \neq 1, 
\end{equation}
then consider the group $G(d')$, where $d = \ell \cdot d'$ and $\ell \nmid d'$.  If $\gs \in G(d')$ is any Galois automorphism satisfying
\begin{equation} \label{restrictedcondition}
\gs |_{\mbq(E[d']) \cap \mbq(\zeta_n)} = \gs_a |_{\mbq(E[d']) \cap \mbq(\zeta_n)},
\end{equation}
then by Corollary \ref{keyacyclicitycor} we may find a lift $\tilde{\gs} \in G(d)$ for which $\tilde{\gs} |_{\mbq(E[d]) \cap \mbq(\zeta_n)} = \gs_a |_{\mbq(E[d]) \cap \mbq(\zeta_n)}$.  Since $S_{E,a,n}'(d) = \emptyset$, there must exist a prime $\ell' \mid d$ with $\tilde{\gs} |_{\mbq(E[\ell'])} = 1$, and by \eqref{gsaisnotidentity}, $\ell' \neq \ell$.  Since $\gs \in G(d')$ satisfying \eqref{restrictedcondition} was arbitrary, it follows that $S_{E,a,n}'(d') = \emptyset$, contradicting that $d$ is the acyclicity level of $G$.  We therefore conclude that $N(\ell) \neq \{ 1 \}$, for any prime $\ell$ dividing $d$, establishing item (2).

To verify item (3) we note that, by Goursat's lemma, there exists a group $\Gamma$ and surjective homomorphisms $\psi : N(\ell) \twoheadrightarrow \Gamma$ and $\psi' : N(d') \twoheadrightarrow \Gamma$ such that
\begin{equation} \label{fiberedproductformofN}
N(d) \simeq N(\ell) \times_{\psi} N(d') := \{ (n,n') \in N(\ell) \times N(d') : \psi(n) = \psi'(n') \}.
\end{equation}
Since $N(\ell) \times_{\psi} N(d') \cap \ker \pi_{d,d'} = \ker \psi \times \{ 1 \}$, it suffices to show that $\ker \psi = \{ 1 \}$.  To see this, fix any ${\boldsymbol{\tau}} = (\tau, \tau') \in \chi^{-1}(\gs_a)$, and note that, if $| \ker \psi | > 1$ then for each ${\boldsymbol{\gs}} = (\gs,\gs') \in (\tau,\tau') N(\ell) \times_{\psi} N(d') = \chi^{-1}(\gs_a)$, there exists $\eta \in \ker \psi - \{ I \}$ such that the element ${\boldsymbol{\nu}} := (\gs,\gs') (\eta,1) \in \chi^{-1}(\gs_a)$ satisfies ${\boldsymbol{\nu}} |_{\mbq(E[\ell])} \neq \gs$ and ${\boldsymbol{\nu}} |_{\mbq(E[d'])} = \gs'$.  Since $S_{E,a,n}'(d) = \emptyset$, this implies that, for each ${\boldsymbol{\gs}} \in \chi^{-1}(\gs_a)$, there must be a prime $\ell' \mid d'$ with ${\boldsymbol{\gs}} |_{\mbq(E[\ell'])} = 1$, implying (as before via Corollary \ref{keyacyclicitycor}) that $S_{E,a,n}'(d') = \emptyset$ and contradicting that $d = \ell d'$ is the acyclicity level of $G$.  Therefore $\ker \psi = \{ 1 \}$, and thus $N(d) \cap \ker \pi_{d,d'} = \{ 1 \}$.  Moreover, it now follows from \eqref{fiberedproductformofN} that $N(\ell) \simeq \Gamma$ is a quotient of $N(d')$, as asserted.  This establishes item (3), finishing the proof.
\end{proof}

We may now reduce the proof of Theorem \ref{secondthm} to a finite computation.  Note that, by (3) of Proposition \ref{acyclicitylemma}, the acyclicity level of any acyclicity group must divide its $\GL_2$-level. 
Define the following collection of open subsets of $\GL_2(\hat{\mbz})$:
\begin{equation} \label{defofmfG}
\mf{G} := \left\{ G \subseteq \GL_2(\hat{\mbz}) : \;  \begin{matrix} G \text{ is an open acyclicity group of composite} \\ \text{$\GL_2$-level that is equal to its acyclicity level}  \end{matrix} \right\}.
\end{equation}
Let $E$ be an elliptic curve over $\mbq$ that belongs to an infinite modular curve family with the property that there exist relatively prime $a, n \in \mbn$ with $C_{E,a,n} = 0$, and assume that there is no prime $\ell \in \{ 2, 3, 5 \}$ for which $\mbq(E[\ell]) \subseteq \mbq(\zeta_n)$ and $\gs_a |_{\mbq(E[\ell])} = 1$.  By  Theorem \ref{abeliandivisionfield}, Remark \ref{acyclicityremark} and \eqref{modularinterpretation}, we must then have that $j_E \in j_{\tilde{G}}\left( X_{\tilde{G}}(\mbq) \right)$ for some $G \in \mf{G}$ (where $\tilde{G}$ is as in \eqref{defofGtilde}) and by Faltings' theorem \cite{faltings}, $\genus(X_{\tilde{G}}) \leq 1$.
Moreover, the set of $\GL_2$-levels of $G \in \mf{G}$ is bounded, as the following corollary of Proposition \ref{acyclicitylemma} shows.  Define the \emph{$\SL_2$-level} of an open subgroup $G \subseteq \GL_2(\hat{\mbz})$ to be the smallest $m \in \mbn$ for which
\[
\ker\left( \SL_2(\hat{\mbz}) \to \SL_2(\mbz/m\mbz) \right) \subseteq G.
\]
It is not difficult to prove that, for any open subgroup $G \subseteq \GL_2(\hat{\mbz})$, the $\SL_2$-level of $G$ divides the $\GL_2$-level of $G$.  In general the quotient $(\text{$\GL_2$-level of }G)/(\SL_2\text{-level of }G) \in \mbn$ is unbounded, as $G \subseteq \GL_2(\hat{\mbz})$ varies over all open subgroups.
\begin{cor} \label{sl2levelcor}
Let $\mf{G}$ be as in \eqref{defofmfG}.  If $G \in \mf{G}$ then the $\SL_2$-level of $\tilde{G}$ is equal to the $\SL_2$-level of $G$, and the $\SL_2$-level of $G$ is equal to the $\GL_2$-level of $G$
\end{cor}
\begin{proof}
Let us denote by $d$ the $\GL_2$-level of $G$ and set $S := G \cap \SL_2(\hat{\mbz})$.  Assume for the sake of contradiction that the $\SL_2$-level of $G$ is strictly less than $d$.  Let $\ell$ be any prime number dividing $d$ that does \emph{not} divide the $\SL_2$-level of $G$ (since $d$ is square-free by \eqref{defofmfG} and Remark \ref{acyclicityremark}, such a prime $\ell$ must exist). Writing $d =: \ell d'$ with $\ell \nmid d'$, we then have that, under the isomorphism of the Chinese remainder theorem,
\[
S(d) \simeq \SL_2(\mbz/\ell\mbz) \times S(d').
\]
It follows that $[G(d),G(d)] \supseteq [S(d),S(d)] \simeq [\SL_2(\mbz/\ell\mbz), \SL_2(\mbz/\ell\mbz)] \times [S(d'),S(d')]$.  Thus, if $\chi : G(d) \to A$ is as in Proposition \ref{acyclicitylemma} and $N(d) := \ker \chi$, we have
\[
N(d) \supseteq [G(d),G(d)] \supseteq [\SL_2(\mbz/\ell\mbz), \SL_2(\mbz/\ell\mbz)] \times \{ 1 \},
\]
contradicting part (3) of Proposition \ref{acyclicitylemma}.  Thus, the $\SL_2$-level of $G$ is equal to $d$, as asserted.

Next, suppose that the $\SL_2$-level of $\tilde{G}$ is less than the $\SL_2$-level of $G$. By \cite[Lemma 2.1]{jonesmcmurdy}, the latter is twice the former, and the proof of \cite[Lemma 2.4]{jonesmcmurdy} shows that, denoting by $\tilde{d}$ the $\SL_2$-level of $\tilde{G}$, we have
\[
S(2\tilde{d}) \simeq \SL_2(\mbz/2\mbz) \times_{\psi} S(\tilde{d}) := \{ (s_2,s_{\tilde{d}}) \in \SL_2(\mbz/2\mbz) \times S(\tilde{d}) : \psi_2(s_2) = \psi_{\tilde{d}}(s_{\tilde{d}}) \},
\]
where $\psi_2 : \SL_2(\mbz/2\mbz) \to \mbz/2\mbz$ and $\psi_{\tilde{d}} : S(\tilde{d}) \to \mbz/2\mbz$ are surjective homomorphisms. Since the common quotient $\mbz/2\mbz$ is cyclic, it follows from \cite[Lemma 1, p. 174]{langtrotterfrob} that, under the isomorphism of the Chinese remainder theorem, $[S(2d),S(2d)] \simeq [\SL_2(\mbz/2\mbz), \SL_2(\mbz/2\mbz)] \times [S(d'),S(d')]$ and thus $N(d) \supseteq [\SL_2(\mbz/2\mbz), \SL_2(\mbz/2\mbz)] \times \{ 1 \}$ in this case as well, again contradicting part (3) of Proposition \ref{acyclicitylemma}.  This proves that the $\SL_2$-levels of $\tilde{G}$ and $G$ agree, concluding the proof of the corollary.
\end{proof}
\begin{cor} \label{levelscor}
Let $\mf{G}$ be as in \eqref{defofmfG}. If $G \in \mf{G}$ and $\genus(G) \leq 1$ then the $\GL_2$-level of $G$ belongs to the set $\left\{ 6, 10, 14, 15, 21, 22, 26, 30, 33, 39, 42 \right\}$.
\end{cor}
\begin{proof}
By Corollary \ref{sl2levelcor}, the (square-free) $\GL_2$-level of $G$ is equal to the $\SL_2$-level of $\tilde{G}$; from Table 2 of \cite{cumminspauli} one can read off the $\SL_2$-levels of all $\tilde{G}$ for which $\genus(X_{\tilde{G}}) \leq 1$.  This yields our list of $\GL_2$-levels.
\end{proof}
Corollary \ref{levelscor}, taken together with a computer calculation using the computational software package {\tt{Magma}} \cite{MAGMA}, shows that
\begin{equation} \label{whatmagmashows}
\forall G \in \mf{G}, \quad \genus(X_{\tilde{G}}) \leq 1 \; \Longrightarrow \; G \; \dot\subseteq \; G_6, 
\end{equation}
where $G_6 \subseteq \GL_2(\hat{\mbz})$ is the group introduced in Section \ref{proofofpartasection}.
Thus,
\[
j_E \in \bigcup_{{\begin{substack} {G \in \mf{G} \\ \genus(X_{\tilde{G}}) \leq 1 } \end{substack}}} j_{\tilde{G}}(X_{\tilde{G}}(\mbq)) 
= j_{G_6}(X_{G_6}(\mbq)).
\]
By \eqref{specialmodularinterpretation}, this finishes the proof of Theorem \ref{secondthm}.  The code used to perform the above-mentioned computation can be found at the following link:
\[
\text{{\tt{https://github.com/ncjones-uic/AcyclicReductions}}}
\]

\begin{remark}
Serre's uniformity question (see \cite[\S 4.3, p. 299]{serre} and also \cite[Conjecture 1.12]{zywina}) asks whether, for each prime $\ell > 37$ and for each elliptic curve $E$ over $\mbq$, one has $\rho_{E,\ell}(G_\mbq) = \GL_2(\mbz/\ell\mbz)$.  Assuming an affirmative answer to this question, we may see that the conclusion of Theorem \ref{secondthm} holds for each elliptic curve $E$ over $\mbq$ whose $j$-invariant does not lie in a certain finite set.
\end{remark}

\end{document}